\documentclass[11pt,twoside]{amsart}

\title[Affine Quermassintegrals and Even Minkowski Valuations]{Af{}f{}ine Quermassintegrals and \\ Even Minkowski Valuations}

\author{Georg C.\ Hofstätter}
\address{Institute of Discrete Mathematics and Geometry, TU Wien, 1040 Vienna, Austria}
\email{georg.hofstaetter@tuwien.ac.at}

\author{Philipp Kniefacz}

\author{Franz E.\ Schuster}
\email{franz.schuster@tuwien.ac.at}

\usepackage[english]{babel}
\usepackage[T1]{fontenc}
\usepackage{lmodern}
\usepackage{amsmath,amsthm,amssymb,amsfonts}
\usepackage[abbrev,nobysame]{amsrefs}
\usepackage{float}
\usepackage{hyperref}
\usepackage[capitalise]{cleveref}
\usepackage[left=4cm,right=4cm,top=3cm,bottom=3cm,includeheadfoot]{geometry}
\usepackage{bbm}
\usepackage{xcolor}
\usepackage{enumerate,enumitem}
\usepackage{tikz-cd}
\usepackage{mathtools}

\hypersetup{colorlinks=true,
	linkcolor=teal,
	citecolor=orange}

\theoremstyle{plain}

\newtheorem{thmintro}{Theorem}
\newtheorem*{thmintro*}{Theorem}
\newtheorem{propintro}[thmintro]{Proposition}
\newtheorem{corintro}[thmintro]{Corollary}
\newtheorem*{corintro*}{Corollary}

\theoremstyle{plain}
\newtheorem{lem}{Lemma}[section]
\newtheorem{prop}[lem]{Proposition}
\newtheorem{thm}[lem]{Theorem}
\newtheorem*{thm*}{Theorem}

\theoremstyle{definition}

\theoremstyle{remark}
\newtheorem{exl}[lem]{Example}
\newtheorem{rem}[lem]{Remark}

\numberwithin{equation}{section}

\crefname{equation}{}{}
\crefname{lem}{Lemma}{Lemmas}
\crefname{thm}{Theorem}{Theorems}
\crefname{thmintro}{Theorem}{Theorems}
\crefname{rem}{Remark}{Remarks}
\crefname{figure}{Figure}{Figures}

\newcommand{\N}{\mathbb{N}}			
\newcommand{\R}{\mathbb{R}}			
\renewcommand{\S}{\mathbb{S}}		
\newcommand{\Gr}{\mathrm{Gr}}		

\newcommand{\K}{\mathcal{K}}		
\newcommand{\Val}{\mathbf{Val}}		

\DeclareMathOperator{\SO}{SO}		
\DeclareMathOperator{\OO}{O}		

\newcommand{\pr}{\mathrm{pr}}								
\newcommand{\pair}[2]{\langle #1,#2 \rangle}				

\DeclareMathOperator{\Cr}{Cr}
\DeclareMathOperator{\Kl}{Kl}
\DeclareMathOperator{\pd}{pd}

\BibSpec{article}{%
	+{}{\PrintAuthors}  		{author}
	+{,}{ \textit}     		{title}
	+{,}{ }             		{journal}
	+{}{ \textbf}       		{volume}
	+{}{ \parenthesize} 		{date}
	+{}{, no. } 		{number}
	+{,}{ }      	      		{conference}
	+{,}{ }      	      		{book}
	+{,}{ }            		{pages}
	+{,}{ }            	 	{note}
	+{,}{ }            	 	{status}
	+{,}{  \texttt } {eprint}
	+{.}{}              {transition}
}

\BibSpec{book}{%
	+{}{\PrintAuthors}  {author}
	+{,}{ \textit}      {title}
	+{,}{ }             {publisher}
	+{,}{ }             {place}
	+{,}{ }             {date}
	+{.}{}              {transition}
}

\BibSpec{incollection}{%
	+{}  {\PrintAuthors}                {author}
	+{,} { \textit}                     {title}
	+{.} { }                            {part}
	+{:} { \textit}                     {subtitle}
	+{,} { \PrintContributions}         {contribution}
	+{,} { \PrintConference}            {conference}
	+{,} { }                            {booktitle}
	+{,} { pp.~}                        {pages}
	+{,}{ }       		{series}
	+{}{ \textbf}       		{volume}
	+{,} { }                            {publisher}
	+{,} { }                 {date}
	+{,} { }                            {status}
	+{,} { \PrintDOI}                   {doi}
	+{,} { available at \eprint}        {eprint}
	+{}  { \parenthesize}               {language}
	+{}  { \PrintTranslation}           {translation}
	+{;} { \PrintReprint}               {reprint}
	+{.} { }                            {note}
	+{.} {}                             {transition}
}

\begin{document}
	
	\begin{abstract}
		It is shown that each continuous even Minkowski valuation on convex bodies of degree $1 \leq i \leq n - 1$ intertwining rigid motions is obtained from convolution of the $i$th~projection function with a unique spherical Crofton distribution. In case of a non-negative distribution, the polar volume of the associated Minkowski valuation gives rise to an isoperimetric inequality which strengthens the classical relation between the $i$th quermassintegral and the volume. This large family of inequalities unifies earlier results obtained for $i = 1$ and $n - 1$. In these cases, isoperimetric inequalities for affine quermassintegrals, specifically the Blaschke--Santal\'o inequality for $i = 1$ and the Petty projection inequality for $i = n - 1$, were proven to be the strongest inequalities. An analogous result for the intermediate degrees is established here. Finally, a new sufficient condition for the existence of maximizers for the polar volume of Minkowski valuations intertwining rigid motions reveals unexpected examples of volume inequalities having asymmetric extremizers.
	\end{abstract}
	
	\maketitle
	\thispagestyle{empty}
	
	\section{Introduction}
	\label{sec:intro}
	
	The classical theory of convex bodies, often referred to as the Brunn--Minkowski theory, arises naturally from the interplay between
	Minkowski addition and volume. The variations of volume, which include surface area and mean width, are the fundamental geometric functionals of the theory and called \emph{quermassintegrals}.
	When viewed as coefficients in Weyl's tube formula, the quermassintegrals appear in differential geometry as integrals of intermediate mean curvatures of convex hypersurfaces. They are also central to various integral geometric formulas, such as the principal kinematic, Crofton's, or Kubota's formula. The latter allows to compute quermassintegrals of a convex body $K$ (a compact, convex set) in $\mathbb{R}^n$, where we assume $n \geq 3$ throughout, from means of its projection functions,
	\begin{align*}
		W_{n-i}(K) = \frac{\kappa_n}{\kappa_i}\int_{\mathrm{Gr}_{i,n}}\!\! V_i(K|E)\,dE, \qquad i = 0, \ldots, n.
	\end{align*}
	The integration here is with respect to the invariant probability measure on the Grassmannian $\mathrm{Gr}_{i,n}$ of $i$-dimensional subspaces of $\mathbb{R}^n$,
	$V_i(K|E)$ is the $i$-dimensional volume of the orthogonal projection of $K$ onto $E \in \mathrm{Gr}_{i,n}$, and $\kappa_i = V_i(B^i)$ with $B^i$ denoting the unit ball in $\mathbb{R}^i$.
	Note that $W_n(K) = \kappa_n$ and $W_0(K) = V_n(K)$.
	
	One of the basic classical inequalities for quermassintegrals of a convex body $K \subseteq \mathbb{R}^n$ with non-empty interior relates $W_{n-i}(K)$, $1 \leq i \leq n - 1$, to the volume of $K$ in the following way
	(see, e.g., \cite{Schneider2014}*{p.\ 401}),
	\begin{equation} \label{WiVninequ}
		W_{n-i}(K)^n \geq \kappa_n^{n-i}\, V_n(K)^i
	\end{equation}
	with equality if and only if $K$ is a Euclidean ball.
	
	
	More recently, conceptually more involved projection inequalities that directly imply (\ref{WiVninequ}) were established by Petty~\cite{Petty1971} for $i = n - 1$ and Lutwak~\cite{Lutwak1985b} for $i = 1, \ldots, n - 2$. In order to explain their improvement of (\ref{WiVninequ}), recall that a convex body $K \subseteq \mathbb{R}^n$ is determined by its support function $h(K,u) = \max\{\pair{u}{x}: x \in K\}$ for $u \in \mathbb{S}^{n-1}$ and that, if $K$ contains the origin in its interior, the polar body of $K$ is given by $K^\circ = \{x \in \mathbb{R}^n: h(K,x) \leq 1\}$. For $K$ with non-empty interior and $1 \leq i \leq n - 1$, the Lutwak--Petty projection inequalities are given by the right-hand side in the chain of inequalities,
	\begin{equation} \label{WVnPiiVnkette}
		W_{n-i}(K)^n \geq \frac{\kappa_n^{n+1}}{\kappa_{n-1}^n} V_n(\Pi_i^\circ K)^{-1} \geq \kappa_n^{n-i}\, V_n(K)^i
	\end{equation}
	with equality on the right if and only if $K$ is an ellipsoid when $i = n - 1$, and a Euclidean ball when $i \leq n - 2$. Here, $\Pi_i^\circ K$ is the polar of the projection body of order $i$ of $K$, defined by $h(\Pi_iK,u) = c_{n,i} W_{n-i}(K|u^\perp)$, where $c_{n,i} > 0$ is usually chosen such that $\Pi_iB^n = \kappa_{n-1} B^n$. The left-hand inequalities in \eqref{WVnPiiVnkette} are due to
	Lutwak \cites{Lutwak1984, Lutwak1985b} with equality if and only if $\Pi_iK$ is a Euclidean ball.
	
	For $i = n - 1$, the quantity $V_n(\Pi_i^\circ K)^{-1/n}$ is (up to a factor) one of the \emph{affine quermassintegrals}, first defined by Lutwak~\cite{Lutwak1984} for a convex body $K \subseteq \R^n$ with non-empty interior by letting $A_n(K) = \kappa_n$, $A_0(K) = V_n(K)$, and for $ 1 \leq i \leq n - 1$,
	\[A_{n-i}(K) = \frac{\kappa_n}{\kappa_i}\!\left (\int_{\mathrm{Gr}_{i,n}}\!\! V_i(K|E)^{-n}\,dE\right)^{\!\!-1/n}.   \]
	Note that while the quermassintegrals $W_{n-i}$, $1 \leq i \leq n - 1$, are merely invariant under rigid motions of $\R^n$, every $A_{n-i}$ is invariant under all volume-preserving affine transformations, as was shown by Grinberg~\cite{Grinberg1991}. A major problem in affine convex geometry, first posed by Lutwak~\cite{Lutwak1988b}, was to obtain a sharp lower bound on $A_{n-i}(K)$, $1 \leq i \leq n - 1$, analogous to \eqref{WiVninequ},
	\begin{equation} \label{Lutwakconjec}
		A_{n-i}(K)^n \geq  \kappa_n^{n-i}\, V_n(K)^i
	\end{equation}
	with equality if and only if $K$ is an ellipsoid. Asymptotic confirmations of \eqref{Lutwakconjec} were obtained in \cites{Dafnis2012, Paouris2013}. Apart from these, only the rank-one cases $i = 1$ (a consequence of the well known Blaschke--Santal\'o inequality) and $i = n - 1$ (the Petty projection inequality) were known until very recently. However, in a landmark paper, Milman and Yehudayoff~\cite{Milman2020} established Lutwak's conjectured inequalities \eqref{Lutwakconjec} giving a unified proof for all $i = 1, \ldots, n - 1$.
	
	Finally, the relation of \eqref{Lutwakconjec} to the chain of inequalities \eqref{WVnPiiVnkette} was settled recently in \cite{Berg2020}, where it was shown that for a convex body $K \subseteq \R^n$ with non-empty interior,
	\begin{equation} \label{PiAffi}
		\frac{\kappa_n^{n+1}}{\kappa_{n-1}^n}  V_n(\Pi_i^\circ K)^{-1} \geq A_{n-i}(K)^n.
	\end{equation}
	Thus, the affine invariant inequality (\ref{Lutwakconjec}) fits in seamlessly in (\ref{WVnPiiVnkette}) and implies the other inequalities of the chain.
	
	
	The aim of this article is to show that \eqref{WVnPiiVnkette} and \eqref{PiAffi} can be extended to a much larger family of inequalities, by proving them not only for the projection body maps but for operators from an infinite dimensional cone. Among the defining characteristics of the maps in this cone are their compatibility with rigid motions and their valuation property, that is, finite (set-theoretic) additivity. More precisely, a map $\phi: \mathcal{K}^n \rightarrow \mathcal{A}$ defined on the set $\mathcal{K}^n$ of convex bodies in $\R^n$ with values in an Abelian semigroup $\mathcal{A}$ is a {\it valuation} if
	\[\phi(K) + \phi(L) = \phi(K \cup L) + \phi(K \cap L)   \]
	whenever $K \cup L$ is convex. The most widely known classical result on real valued valuations, that is, $\mathcal{A} = \R$, is the classification of all continuous rigid motion invariant valuations by Hadwiger~\cite{Hadwiger1957} as precisely the linear combinations of the quermassintegrals, see \textbf{\cite{Klain1997}} and \cite{Schneider2008} for more information on the history of this result and its transformational impact on integral geometry.
	
	Inspired by Hadwiger's theorem, different choices for the semigroup $\mathcal{A}$ became an important focus of interest. In particular, over the last two decades \emph{Minkowski valuations}, where $\mathcal{A} = \mathcal{K}^n$ and addition on $\mathcal{K}^n$ is Minkowski addition, received widespread attention. This line of research has its origins in the seminal work of Ludwig~\cites{Ludwig2002b, Ludwig2005} and, first, was mainly concerned with classifying continuous Minkowski valuations compatible with \emph{linear} transformations \cites{Abardia2012b, Abardia2011b, Abardia2018b, Ludwig2010b, Schuster2012, Wannerer2011}. In this case, recent results show that these valuations form a cone generated by finitely many maps (such as the projection body map $\Pi_{n-1}$). In contrast, results by Kiderlen~\cite{Kiderlen2006} and from \cite{Schuster2007} imply that the cone of all translation invariant continuous Minkowski valuations which merely commute with $\mathrm{SO}(n)$ (like the operators $\Pi_i$ when $i = 1, \ldots, n - 2$) is infinite-dimensional. The natural problem to also obtain a precise description of this cone has yet to be solved. By McMullen~\cite{McMullen1977}, only integer degrees of homogeneity $0 \leq i \leq n$ can occur, with $i = 0$ and $i = n$ being trivial. In \cite{Kiderlen2006} and \cite{Schuster2007}, convolution representations were established for the cases $i = 1$ and $i = n - 1$, respectively. For \emph{even} Minkowski valuations these results were subsequently generalized in \cite{Schuster2010} and \cite{Schuster2015} to all intermediate degrees under an additional smoothness assumption. With our first result, we are able to remove this strong regularity condition.
	
	\begin{propintro} \label{mthm:repEvenMinkVal}
		Suppose that $\Phi_i: \mathcal{K}^n \rightarrow \mathcal{K}^n$ is a continuous, translation invariant, and $\mathrm{SO}(n)$ equivariant Minkowski valuation of degree $i \in \{1, \ldots, n - 1\}$. \linebreak
		If $\Phi_i$ is even, then there exists an $\mathrm{O}(i) \times \mathrm{O}(n-i)$ invariant distribution $\delta$ on $\mathbb{S}^{n-1}$ uniquely determined by the property that for every $K \in \mathcal{K}^{n}$,
		\begin{align*}
			h(\Phi_i K, \cdot) = V_i(K|\,\cdot\,) \ast \delta.
		\end{align*}
		The distribution $\delta$ is called the \emph{spherical Crofton distribution} of $\Phi_i$.
	\end{propintro}
	
	The convolution here is induced from the compact group $\mathrm{SO}(n)$ by identifying $\mathbb{S}^{n-1}$ and $\mathrm{Gr}_{i,n}$ with homogeneous spaces with respect to $\mathrm{SO}(n)$ (see~\cref{sec:bgConv}).
	
	
	The proof of \cref{mthm:repEvenMinkVal} relies on results about generalized valuations by Alesker and Faifman~\cite{Alesker2014} and the techniques to obtain the earlier versions of \cref{mthm:repEvenMinkVal} from \cite{Schuster2010} and \cite{Schuster2015}. Examples of spherical Crofton distributions and more details on the class of distributions that appear for different degrees will be given in \cref{sec:prfIneq} and in \cref{sec:exExtr}.
	
	In recent years, several isoperimetric-type inequalities involving projection body maps (of arbitrary degree) were shown to hold, in fact, for much larger classes of Minkowski valuations compatible with rigid motions (see \cites{Alesker2011, Berg2014, Haberl2019, Hofstaetter2021, Parapatits2012, Schuster2007}).
	Based on \cref{mthm:repEvenMinkVal}, we establish with our main result a significant extension of inequalities \eqref{WVnPiiVnkette} and \eqref{PiAffi} to all even Minkowski valuations admitting a spherical Crofton distribution which is non-negative (and, thus, a spherical Crofton \emph{measure}).
	
	\begin{thmintro} \label{mthm:ineqEvenMinkVal} Suppose that the spherical Crofton distribution of an even Minkowski valuation $\Phi_i : \mathcal{K}^n \rightarrow \mathcal{K}^n$ of degree $i$, $1 \leq i \leq n - 1$, is non-negative and normalized such that $\Phi_i B^n = \kappa_{n-1}B^n$. If $K \in \mathcal{K}^n$ has non-empty interior, then
		\begin{align}\label{eq:mthmIneq}
			W_{n-i}(K)^{n} \geq \frac{\kappa_n^{n+1}}{\kappa_{n-1}^n} V_n(\Phi_i^\circ K)^{-1} \geq  A_{n-i}(K)^{n}.
		\end{align}
		There is equality in the left hand inequality if and only if $\Phi_iK$ is a Euclidean ball. Equality holds in the right hand inequality if and only if $K$ is of constant $i$-brightness or $i = 1$ and $\Phi_1 = \frac{\kappa_{n-1}}{2}(-\mathrm{Id} + \mathrm{Id})$ or $i = n - 1$ and $\Phi_{n-1} = \Pi_{n-1}$.
	\end{thmintro}
	
	Let us note that (as we shall show) the assumption that the spherical Crofton distribution is non-negative is only necessary for the right hand inequality. In fact, the left hand inequality, including its equality conditions, holds for all continuous Minkowski valuations which are translation in- and $\mathrm{SO}(n)$ equivariant, mapping bodies with non-empty interior to such bodies. Next, we want to point out that when $i = 1$ or $i = n - 1$ \emph{all} non-negative spherical measures which are $\mathrm{O}(n-1)$ invariant are spherical Crofton measures. It is an open problem, whether the same is true also for the degrees $2 \leq i \leq n - 2$. Finally, the special cases $i = 1$ and $i = n - 1$ of \cref{mthm:ineqEvenMinkVal} were previously obtained in \cite{Hofstaetter2021} and \cite{Haberl2019}, respectively. Partial results for the intermediate degrees, without the equality cases for the right hand inequality, were obtained in \cite{Berg2020}. However, we will see in \cref{sec:prfIneq} that \cref{mthm:ineqEvenMinkVal} provides a significant extension of these earlier results.
	
	An immediate consequence of \cref{mthm:ineqEvenMinkVal} and inequality \eqref{Lutwakconjec} of Milman and Yehudayoff is the solution to the following isoperimetric problem.
	
	\begin{corintro}\label{mcor:maxEvenMinkVal}
		Suppose that $\Phi_i: \mathcal{K}^n \rightarrow \mathcal{K}^n$ is a continuous, translation invariant, and $\mathrm{SO}(n)$ equivariant Minkowski valuation of a given degree $i \in \{1, \ldots, n - 1\}$.
		If $\Phi_i$ is non-trivial, even and its spherical Crofton distribution non-negative, then, among $K \in \mathcal{K}^n$ with non-empty interior,
		\begin{equation} \label{volprodphii}
			V_n(\Phi_i^\circ K) V_n(K)^i
		\end{equation}
		is maximized by Euclidean balls. If $i = 1$ and $\Phi_1 = c(-\mathrm{Id} + \mathrm{Id})$ or $i = n - 1$ and $\Phi_{n-1} = c\Pi_{n-1}$ for some $c > 0$, then $K$ is a maximizer if and only if it is an ellipsoid. Otherwise, Euclidean balls are the only maximizers.
	\end{corintro}
	
	
	
	Let us emphasize that the existence of extremals for \eqref{volprodphii} is a-priori not clear. In case of continuous affine invariant functionals, this follows easily from compactness
	(see, e.g., \cite{Schneider2014}*{Chapter 10}). However, as was recently discovered in \cite{Hofstaetter2021}, there exists a continuous Minkowski valuation $\Phi_1$ of degree 1 compatible with rigid motions such that $V_n(\Phi_1^\circ K) V_n(K)$ is unbounded. As this somewhat surprising example is \emph{not even}, we consider for our final result -- a sufficient condition for the existence of maximizers of volume products of the form \eqref{volprodphii} -- also Minkowski valuations that are not necessarily even. To this end, we require the following counterpart of \cref{mthm:repEvenMinkVal} for such valuations. Here, $S_i(K, \cdot)$ denotes the area measure of order $i$ of $K \in \K^n$, $0 \leq i \leq n-1$ (see, e.g., \cite{Schneider2014}*{Ch.~4}).
	
	
	\begin{thm}[\cites{Dorrek2017b, Schuster2018}]\label{thm:repGenFct}
		If $\Phi_i: \mathcal{K}^n \rightarrow \mathcal{K}^n$ is a continuous, translation invariant, and $\mathrm{SO}(n)$ equivariant Minkowski valuation of a given degree $i \in \{1, \ldots, n - 1\}$, then there exists a unique $\mathrm{SO}(n - 1)$ invariant $f \in L^1(\mathbb{S}^{n-1})$ with center of mass at the origin such that for every $K \in \mathcal{K}^n$,
		\begin{equation} \label{convrep}
			h(\Phi_i K,\cdot) = S_i(K,\cdot) \ast f.
		\end{equation}
		\noindent The function $f$ is called the \emph{generating function} of $\Phi_i$.
	\end{thm}
	Let us point out that it was shown in \cite{Brauner2023} that in certain cases, the generating function $f$ is actually much more regular.

	The aforementioned sufficient condition can now be stated in terms of generating functions. We also include in our result an immediate application of the condition that uncovers a phenomenon not seen before.
	
	\begin{thmintro} \label{mthm:existExtr} Suppose that $\Phi_i: \mathcal{K}^n \rightarrow \mathcal{K}^n$ is a continuous, translation invariant, and $\mathrm{SO}(n)$ equivariant Minkowski valuation of a given degree
		$i \in \{1, \ldots, n - 1\}$. If the generating function of $\Phi_i$ is a sum of two generating functions one of which is bounded from below by a positive constant, then
		\[V_n(\Phi_i^\circ K) V_n(K)^i\]
		attains a maximum on convex bodies $K \in \mathcal{K}^n$ with non-empty interior.
		Moreover, for $i = 1$, there exist Minkowski valuations $\Psi_1: \mathcal{K}^n \rightarrow \mathcal{K}^n$ such that the maximizers of $V_n(\Psi_1^\circ K) V_n(K)$ are different from Euclidean balls.
	\end{thmintro}
	
	Let us point out that an even Minkowski valuation with a positive generating function need not have a non-negative spherical Crofton distribution (cf.\ \cref{thm:monotExMax}). Let us also note that the positivity condition on generating functions in \cref{mthm:existExtr} is not very restrictive. Indeed, generating functions of Minkowski valuations of degree $i = n - 1$ are all non-negative. Moreover, strictly positive support functions of convex bodies of revolution generate a large class of examples of Minkowski valuations of arbitrary degree $1 \leq i \leq n - 1$.
	
	In addition, our proof of \cref{mthm:existExtr} shows that any example of a Minkowski valuation $\Phi_i$ with unbounded volume product $V_n(\Phi_i^\circ K) V_n(K)^i$ yields an entire cone with apex at $\Phi_i$ of Minkowski valuations $\Psi_i$ such that the maximizers of $V_n(\Psi_i^\circ K) V_n(K)^i$ exist and are different from Euclidean balls. In the final section of the article, we will also comment on non-polar versions of \cref{mthm:existExtr}.
	
	\newpage
	
	\section{Background material on Harmonic Analysis}\label{sec:bgConv}
	In this section we recall some well known facts about generalized functions and the notion of convolution of functions, measures, and distributions on the homogeneous spaces $\mathbb{S}^{n-1}$ and $\Gr_{i,n}$. For the first reading, this section may be skipped, we will mostly need equations \eqref{eq:convGrSph}, \eqref{eq:ConvSphSphMeas} and \eqref{eq:ConvSphMeasSph}. As general references for the material presented here we recommend the articles \cite{Grinberg1999} and \cite{Schuster2018}.
	
	Let us first consider a general smooth manifold $M$. (In what follows, $M$ will be either Euclidean space $\R^n$, the Euclidean unit sphere $\S^{n-1}$, the Lie group $\SO(n)$, or the Grassmannian $\Gr_{i,n}$.) As usual $C^{\infty}(M)$ denotes the space of all smooth functions on $M$ equipped with the Fr\'echet space topology of uniform convergence of any finite number of derivatives on each compact subset of $M$. If $X$ is a Banach space, then the Fr\'echet space $C^{\infty}(M,X)$ of all infinitely differentiable functions on $M$ with values in $X$ is defined in a similar way.
	
	We denote by $C^{-\infty}(M)$ the space of distributions on $M$, that is, the topological dual space of the space $C_c^{\infty}(M)$ of all functions in $C^{\infty}(M)$ with compact support.
	Note that $C^{-\infty}(M)$ often denotes the space of generalized functions, that is, continuous linear functionals on $C_c^{\infty}(M,|\Lambda|(M))$, where $|\Lambda|(M)$ is the 1-dimensional space of smooth densities on $M$. However, the choice of a Riemannian metric on $M$ induces an isomorphism between the spaces of distributions and generalized functions. In the following, we make use of this identifications and write $C^{-\infty}(M)$ for the space of distributions equipped with the topology of weak convergence. The canonical bilinear pairing on $C_c^{\infty}(M) \times C^{-\infty}(M)$ is denoted by $\pair{\cdot}{\cdot}_{C^{-\infty}}$.
	
	Letting $\mathcal{M}(M)$ denote the space of signed finite Borel measures on $M$, note that every $\mu \in \mathcal{M}(M)$ defines a distribution $\nu_{\mu}$ by
	\begin{align*}
		\pair{f}{\nu_{\mu}}_{C^{-\infty}} = \int_{M} f(x)\,d\mu(x), \qquad f \in C_c^{\infty}(M).
	\end{align*}
	We will use the continuous linear injection $\mu \mapsto \nu_{\mu}$ to identify $\mathcal{M}(M)$ with a (dense) subspace of $C^{-\infty}(M)$. In the same way, the spaces $C^{\infty}(M)$, $C(M)$, and $L^2(M)$ can be identified with subspaces of $C^{-\infty}(M)$ and we will often not distinguish between a function or measure and its associated distribution.
	
	Let $M$ now be one of the compact smooth manifolds $\SO(n)$, $\S^{n-1}$, or $\Gr_{i,n}$. Then for a subgroup $H$ of $\SO(n)$, the left-action $l_{\vartheta}$ of $\vartheta \in H$ on $f \in C^{\infty}(M)$ is given by
	\begin{align*}
		(l_{\vartheta} f)(x) = f(\vartheta^{-1}x), \qquad x \in M.
	\end{align*}
	It extends to distributions as follows: For $\vartheta \in H$ and $\nu \in C^{-\infty}(M)$, we set
	\begin{align*}
		\pair{l_{\vartheta} \nu}{f}_{C^{-\infty}} = \pair{\nu}{l_{\vartheta^{-1}} f}_{C^{-\infty}}, \qquad f \in C^{\infty}(M).
	\end{align*}
	Note that this definition is compatible with the embeddings of functions and measures in $C^{-\infty}(M)$ and that if $\nu$ is a measure on $M$, then $l_{\vartheta} \nu$ is just the image measure	of $\nu$ under the rotation $\vartheta$. 
	
	When the manifold $M$ is the Lie group $\SO(n)$, there is also a natural right-action $r_{\vartheta}$ of an element $\vartheta$ of a subgroup $H$ of $\SO(n)$ on a distribution $\nu$ (or function or measure) on $\SO(n)$ given by
	\begin{align*}
		\pair{r_{\vartheta} \nu}{f}_{C^{-\infty}} = \pair{\nu}{r_{\vartheta^{-1}} f}_{C^{-\infty}}, \qquad f \in C^{\infty}(\SO(n)),
	\end{align*}
	and $r_{\vartheta}f(\tau) = f(\tau\vartheta)$, $\tau, \vartheta \in \SO(n)$.
	
	A function, measure, or distribution on $\SO(n)$ is called left or right $H$-invariant, respectively, if it is invariant with respect to the left- or right-action of any element from $H$, respectively.
	
	Not only left and right multiplication in $\SO(n)$ induce natural left- and right-actions on functions, measures, and distributions, but also the inversion map on $\SO(n)$ gives rise
	to a natural involution on functions and distributions as follows: For $f \in C^{\infty}(\SO(n))$, the function $\widehat{f} \in C^{\infty}(\SO(n))$ is defined by
	\begin{align*}
		\widehat{f}(\vartheta) = f(\vartheta^{-1}), \qquad \vartheta \in \SO(n),
	\end{align*}
	and, for a distribution $\nu$ on  $\SO(n)$, we define $\widehat{\nu} \in C^{-\infty}(\SO(n))$ by
	\begin{align*}
		\pair{\widehat{\nu}}{f}_{C^{-\infty}} = \pair{\nu}{\widehat f}_{C^{-\infty}}, \qquad  f \in C^{\infty}(\SO(n)).
	\end{align*}
	In particular, it follows that
	\begin{align}\label{eq:pairingDistrHat}
		\pair{\nu}{f}_{C^{-\infty}} = \pair{\widehat{\nu}}{\widehat{f}}_{C^{-\infty}} \quad \text{ and } \quad \widehat{l_{\vartheta} f} = r_{\vartheta} \widehat{f},
	\end{align}
	where the second identity implies that the inversion map interchanges left- and right $H$-invariance of functions and distributions.
	
	\medskip
	
	Next, the convolution of $\nu \in C^{-\infty}(\SO(n))$ and $f \in C^\infty(\SO(n))$ is defined by
	\begin{align}\label{eq:convDistrib}
		(\nu \ast f)(\eta) = \pair{\nu}{l_\eta \widehat{f}}_{C^{-\infty}}, \quad \eta \in \SO(n),
	\end{align}
	in particular, $\pair{\nu}{f} = (\nu \ast \widehat{f})(\mathrm{id})$. Note that $\nu \ast f \in C^\infty(\SO(n))$. The convolution of distributions $\nu, \delta \in C^{-\infty}(\SO(n))$ is defined by associativity and \eqref{eq:convDistrib}, that is,
	\begin{align}\label{eq:defConvDistribByAssoc}
		(\nu \ast \delta) \ast f = \nu \ast (\delta \ast f)
	\end{align}
	for all $f \in C^\infty(\SO(n))$. Using \eqref{eq:convDistrib} and the comment beneath it, it is not hard to check that for measures $\mu, \sigma \in \mathcal{M}(\SO(n))$, definition \eqref{eq:defConvDistribByAssoc} is equivalent to
	\begin{align*}
		\int_{\SO(n)} f(\tau) d(\mu \ast \sigma)(\tau) = \int_{\SO(n)}\int_{\SO(n)} f(\eta\vartheta)d\mu(\vartheta)d\sigma(\eta), \quad f \in C^\infty(\SO(n)).
	\end{align*}
	Let us point out that $l_{\vartheta} (\nu \ast \delta) = (l_\vartheta \nu) \ast \delta$, $r_{\vartheta}(\nu \ast \delta) = \nu \ast (r_{\vartheta} \delta)$ and $\nu \ast (l_{\vartheta^{-1}}\delta) = (r_\vartheta \nu) \ast \delta$. In particular, if $\nu$ is left $H$-invariant for some subgroup $H$ of $\SO(n)$, so is $\nu \ast \delta$, and if $\delta$ is right $H$-invariant, so is $\nu \ast \delta$. If $\nu$ is right $H$-invariant, we may assume that $\delta$ is left $H$-invariant, and vice versa. Moreover, $\widehat{\nu \ast \delta} = \widehat{\delta} \ast \widehat{\nu}$.

	For later reference, we note an easy continuity property of convolution.
	\begin{lem}\label{lem:convergenceConv}
		If $\nu_j \in C^{-\infty}(\SO(n))$, $j \in \N$, converge to $\nu \in  C^{-\infty}(\SO(n))$, then
		$\nu_j \ast \delta$ converge to $\nu \ast \delta$ and $\delta \ast \nu_j$ converge to $\delta \ast \nu$ for every $\delta \in  C^{-\infty}(\SO(n))$.
	\end{lem}

	Letting $\bar{e} \in \S^{n-1}$ be fixed but arbitrary, we can consider $\S^{n-1}$ as homogeneous space $\SO(n)/\SO(n-1)$, where we denote by $\SO(n-1)$ the stabilizer of $\bar{e}$. Similarly, by fixing a subspace $\overline{E} \in \Gr_{i,n}$, we identify $\Gr_{i,n} = \SO(n)/\mathrm{S}(\OO(i) \times \OO(n-i))$, where $\mathrm{S}(\OO(i) \times \OO(n-i))$ denotes the stabilizer of $\overline{E}$.
	
	In the following, let $H$ be one of the subgroups $\SO(n-1)$ or $\mathrm{S}(\OO(i) \times \OO(n-i))$, $1 \leq i \leq n-1$, of $\SO(n)$, and denote by $\pr_H: \SO(n) \to \SO(n)/H$ the canonical projection.
	
	Then the pullback $(\pr_H)^\ast: C^\infty(\SO(n)/H) \to C^\infty(\SO(n))$ and the pushforward $(\pr_H)_\ast: C^\infty(\SO(n)) \to C^\infty(\SO(n)/H)$ are defined by
	\begin{align*}
		(\pr_H)^\ast f(\eta) = f(\pr_H \eta) \quad \text{ and } \quad (\pr_H)_\ast g(\eta H) = \int_H g(\eta \tau) d\tau, \quad \eta \in \SO(n),
	\end{align*}
	for all $f \in C^\infty(\SO(n)/H)$ and $g \in C^\infty(\SO(n))$, where the measure on $H$ is the unique invariant probability measure. They are naturally extended to distributions by duality, that is, for $\nu \in C^{-\infty}(\SO(n)/H)$ and $\delta \in C^{-\infty}(\SO(n))$,
	\begin{align*}
		\pair{(\pr_H)^\ast \nu}{g} = \pair{\nu}{(\pr_H)_\ast g} \quad \text{ and } \quad \pair{(\pr_H)_\ast \delta}{f} = \pair{\nu}{(\pr_H)^\ast f}
	\end{align*}
	for all $f \in C^\infty(\SO(n)/H)$ and $g \in C^\infty(\SO(n))$. Clearly, $(\pr_H)^\ast \nu$ is right $H$-invariant, while any right $H$-invariant distribution on $\SO(n)$ can be represented as a pullback of a distribution on $\SO(n)/H$, using that $(\pr_H)_\ast \circ (\pr_H)^\ast = \mathrm{id}$.
	
	We can now use pullback and pushforward to define convolution on $\SO(n)/H$. Indeed, for $\nu, \delta \in C^{-\infty}(\SO(n)/H)$, the convolution of $(\pr_H)^\ast \nu$ and $(\pr_H)^\ast \delta$ is right $H$-invariant and therefore descends to a distribution on $\SO(n)/H$. This operation is again denoted by $\ast$. In a similar way, a distribution $\nu$ on $\SO(n)/H_1$ can be convoluted with a distribution $\delta$ on $\SO(n)/H_2$ for two different subgroups $H_1, H_2$ of $\SO(n)$, yielding a distribution $\nu \ast \delta$ on $\SO(n)/H_2$. Moreover, $\nu \ast \delta$ will be left $H_1$-invariant.
	
	In the following, we will need explicit formulas for the convolution in three cases: First, for $f \in C(\Gr_{i,n})$ and $\mu \in \mathcal{M}(\S^{n-1})$,
	\begin{align}\label{eq:convGrSph}
		(f \ast \mu)(u) = \int_{\Gr_{i,n}} f(\vartheta_u F) d\widehat{\mu}(F), \quad u \in \S^{n-1},
	\end{align}
	where $\vartheta_u \in \SO(n)$ is chosen from $\pr_{\SO(n-1)}^{-1}(u)$, that is, $\vartheta_u \bar{e} = u$. Second, for $g \in C(\S^{n-1})$ and left $\SO(n-1)$-invariant $\mu \in \mathcal{M}(\S^{n-1})$,
	\begin{align}\label{eq:ConvSphSphMeas}
		(g \ast \mu)(u) = \int_{\S^{n-1}} g(\vartheta_uv)d\mu(v), \quad u \in \S^{n-1}
	\end{align}
	Note here that as $\mu$ is left $\SO(n-1)$-invariant, $\widehat{\mu} = \mu$.
	
	Finally, for $g = \bar{g}(\pair{\cdot}{\bar e}) \in C(\S^{n-1})$ and $\mu \in \mathcal{M}(\S^{n-1})$,
	\begin{align}\label{eq:ConvSphMeasSph}
		(\mu \ast g)(u) = \int_{\S^{n-1}} g(\vartheta_u^{-1}v)d\mu(v) = \int_{\S^{n-1}} \bar{g}(\pair{u}{v})d\mu(v), \quad u \in \S^{n-1}.
	\end{align}

	\section{Proof of \texorpdfstring{\cref{mthm:repEvenMinkVal}}{Proposition A}}\label{sec:prfIneq}
	In this section, we will first recall necessary background from valuation theory, which we then use to deduce \cref{mthm:repEvenMinkVal}. The main observation here is that for every continuous, translation invariant Minkowski valuation $\Phi: \K^n \to \K^n$, we can define a real-valued valuation $\varphi: \K^n \to \R^n$ by
	\begin{align*}
		\varphi(K) = h(\Phi K, \bar{e}), \quad K \in \K^n.
	\end{align*}
	Clearly, $\varphi$ is continuous and translation invariant. If $\Phi$ is additionally $\SO(n)$ equivariant, then $\varphi$ is invariant under $\SO(n-1)$, the stabilizer of $\bar e$. Moreover, in this case, $\varphi$ determines $\Phi$ completely, as $h(\Phi K, \eta \bar e) = \varphi(\eta^{-1}K)$, $\eta \in \SO(n)$. Consequently, we will use results for the associated real-valued valuation to deduce analogous statements for Minkowski valuations.
	
	To this end, denote by $\Val(\R^n)$ the space of continuous, translation invariant, real valued valuations on $\R^n$, and by $\Val_i^{\pm}(\R^n)$ its subspaces of even resp.\ odd valuations homogeneous of degree $i \in \{0, \dots, n\}$. The space $\Val(\R^n)$ becomes a Banach space when equipped with the topology of uniform convergence on compact subsets of $\K^n$.
	
	By a result of Klain~\cite{Klain2000}, every even valuation in $\Val_i(\R^n)$ gives rise to a continuous function on $\Gr_{i,n}$. This identification, called the Klain map $\Kl_i: \Val_i^+(\R^n) \to C(\Gr_{i,n})$, is defined implicitly by the equation
	\begin{align*}
		\varphi(K) = (\Kl_{i} \varphi)(E) V_i^E(K), \quad K \in \K(E),
	\end{align*}
	using that, by a result by Hadwiger~\cite{Hadwiger1957}, every $\varphi \in \Val_i^+(\R^n)$ coincides on convex bodies contained in an $i$-dimensional subspace $E$ (denoted by $\K(E)$) with a constant multiple $(\Kl_{i} \varphi)(E)$ of the volume $V_i^E$ on $E$.
	
	\medskip
	
	It is sometimes more convenient to work with the subspace $\Val^\infty(\R^n)$ (and similarly $\Val^{\pm,\infty}_i(\R^n)$) of smooth valuations, that is, those $\varphi \in \Val(\R^n)$ for which the map $\mathrm{GL}(\R^n) \ni g \mapsto \varphi(g^{-1}(\cdot)) \in \Val(\R^n)$ is infinitely differentiable as a map from a smooth manifold to a Banach space. The space $\Val^\infty(\R^n)$ becomes a graded algebra, when endowed with the Alesker product (see, e.g., \cite{Alesker2004b}). See, e.g., \cites{Alesker2006,Alesker2007,Alesker2008,Alesker2006b} for a thorough study of smooth valuations.
	
	Further, in \cite{Alesker2014}, the space $\Val_i^{\pm,-\infty}(\R^n)$ of $i$-homogeneous, even resp.\ odd generalized valuations was defined by
	\begin{align*}
		\Val_i^{\pm,-\infty}(\R^n) = \left(\Val_{n-i}^{\pm, \infty}(\R^n)\right)^\ast \otimes \mathrm{Dens}(\R^n),
	\end{align*}
	where we denote by $(\cdot)^\ast$ the topological dual space and by $\mathrm{Dens}(\R^n)$ the one-dimensional space of densities on $\R^n$. As we tacitly fixed the standard Euclidean structure on $\R^n$, we can identify $\Val_i^{\pm, -\infty}(\R^n) \cong (\Val_{n-i}^{\pm, \infty}(\R^n))^\ast$. The standard pairing of $\Val_i^{\pm, -\infty}(\R^n)$ and $\Val_{n-i}^\infty(\R^n)$ is denoted by $\pair{\cdot}{\cdot}_{\Val^{-\infty}}$. 
	
	The Alesker product gives rise to a non-degenerate pairing on $\Val^\infty(\R^n)$, called the (Alesker-)Poincar\'e duality, and, thus, we can embed $\Val_i^{\pm, \infty}(\R^n)$ into $\Val_i^{\pm, -\infty}(\R^n)$. Using an extension of the product in \cite{Alesker2011b}, allowing to take the product of a continuous and a smooth valuation, this embedding can be extended to yield a map $\widetilde{\pd}$,
	\begin{align*}
		\Val^\infty(\R^n) \subset \Val(\R^n) \xhookrightarrow{\quad\widetilde{\pd}\quad} \Val^{-\infty}(\R^n).
	\end{align*}
	The same is true for $\Val_i(\R^n)$ and $\Val_i^{\pm}(\R^n)$. Let us note that the product was further extended to a partial product on $\Val^{-\infty}(\R^n)$ in \cite{Alesker2009}.

	\medskip
	
	In the following, we will need an extension of the Klain map to the space of generalized valuations, established in \cite{Alesker2014}. This extension goes hand in hand with an extension of the Crofton map $\Cr_i: C(\Gr_{i,n}) \to \Val_i^+(\R^n)$, which is defined for $f \in C(\Gr_{i,n})$ by
	\begin{align}\label{eq:croftCont}
		(\Cr_i f) (K) = \int_{\Gr_{i,n}} V_i(K|E) f(E)dE, \quad K \in \K^n.
	\end{align}
	A direct calculation shows that $\Kl_i \circ \Cr_{i} = C_i$, where $C_i: C(\Gr_{i,n}) \to C(\Gr_{i,n})$ denotes the cosine transform on $\Gr_{i,n}$ (see, e.g., \cite{Alesker2004} for the definition).
	
	In the following, we write $f^\perp$ for the function $E \mapsto f(E^\perp)$, $E \in \Gr_{n-i,n}$, $f \in C(\Gr_{i,n})$, and correspondingly for measures.
	
	\begin{thm}[\cite{Alesker2014}]\label{thm:AlFExtCrKl}
		Let $1 \leq i \leq n-1$. The Crofton map $\Cr_i$ and the Klain map $\Kl_i$ can be extended uniquely by continuity to
		\begin{align*}
			\widetilde {\Cr_{i}}: C^{-\infty}(\Gr_{i,n}) \rightarrow \Val^{+, -\infty}_i, \quad \widetilde {\Kl_{i}}: \Val^{+, -\infty}_i \rightarrow C^{-\infty}(\Gr_{i,n}),
		\end{align*}
		where $\widetilde \Cr_i$ is surjective and $\widetilde \Kl_i$ is injective. Moreover, $\widetilde {\Cr_{i}}$ is adjoint to $\Kl_{n-i}$, 
		\begin{align}\label{eq:thmAlFExtCrKlAdjRel}
			\langle\widetilde {\Cr_{i}} (\psi), \varphi \rangle_{\Val^{-\infty}} = \langle \psi, (\Kl_{n-i} \varphi)^\perp \rangle_{C^{-\infty}},
		\end{align}
		for $\psi \in C^{-\infty}(\Gr_{i,n})$ and $\varphi \in \Val^{+, \infty}_{n-i}$, and the extensions satisfy $\widetilde\Kl_i \circ \widetilde\Cr_i=C_i$.
	\end{thm}
	As a consequence of \cref{thm:AlFExtCrKl}, Alesker and Faifman~\cite{Alesker2014} deduced the following representation for any $\varphi \in \Val_i^+(\R^n)$ similar to \eqref{eq:croftCont}. Here, a body $K \in \K^n$ is said to be of class $C^\infty_+$, if its boundary hypersurface is a smooth submanifold of $\R^n$ with everywhere positive Gauss curvature (see, e.g., \cite{Schneider2014}*{Sec.~2.5}).
	\begin{lem}[\cite{Alesker2014}*{Lem.~4.7}]\label{lem:AlFRepCri}
		Let $1 \leq i \leq n-1$ and suppose that $\varphi \in \Val_i^{+}$ and $\widetilde \pd\, \varphi = \widetilde \Cr_i \psi$ for some $\psi \in C^{-\infty}(\Gr_{i,n})$. If $K \in \K^n$ is of class $C^\infty_+$, then
		\begin{align}\label{eq:lemAlFRepCriClaim}
			\varphi(K) = \langle \psi, V_i(K|\cdot) \rangle_{C^{-\infty}}.
		\end{align}
	\end{lem}
	
	We can now apply \cref{thm:AlFExtCrKl} and \cref{lem:AlFRepCri} to deduce \cref{mthm:repEvenMinkVal}.
	
	\begin{proof}[Proof of \cref{mthm:repEvenMinkVal}]
		If $\Phi_i$ is an even, continuous, $\SO(n)$ equivariant, and translation invariant Minkowski valuation, then the associated real-valued valuation $\varphi_i \in \Val^{+}_i$ is $\SO(n-1)$ invariant. Using the extended Poincar\'e duality map $\widetilde {\pd}$, we may embed $\varphi_i$ into the space $\Val_i^{+,-\infty}$. Thus, by \cref{thm:AlFExtCrKl}, there exists a Crofton distribution
		$\psi \in C^{-\infty}(\Gr_{i,n})$ such that
		\begin{align*}
			\widetilde {\pd}\, \varphi_i = \widetilde {\Cr_{i}} \psi.
		\end{align*}
		Since both $\widetilde {\pd}$ and $\widetilde {\Cr_{i}}$ are $\SO(n)$ equivariant, we may assume without loss of generality that $\psi$ is $\SO(n-1)$ invariant.
		If $K \in \mathcal{K}^n$ is of class $C^\infty_+$, then we are now able to compute $h(\Phi_i K, \eta \bar e)$ for $\eta \in \mathrm{SO}(n)$ by applying \cref{lem:AlFRepCri}, \eqref{eq:pairingDistrHat}, and \eqref{eq:convDistrib},
		\begin{align*}
			h(\Phi_i K, \eta \bar e)        &= \varphi_i(\eta^{-1} K) = \pair{\psi}{l_{\eta^{-1}} V_i(K |\, \cdot\,) }_{C^{-\infty}}  \\
			&= \pair{\widehat{\psi}}{r_{\eta^{-1}} \widehat{V_i(K|\,\cdot\,)} }_{C^{-\infty}} \\
			&= (V_i(K|\,\cdot\,) \ast \hat \psi) (\eta \bar e).
		\end{align*}
		Setting $\delta = \widehat{\psi}$ and noting that the $\SO(n-1)$ invariance of $\psi \in C^{-\infty}(\Gr_{i,n})$ implies the $\OO(i) \times \OO(n-i)$ invariance of $\delta \in C^{-\infty}(\S^{n-1})$ (see \cite{Schuster2015}*{Prop.~A}), we obtain $h(\Phi_i K, \cdot) = V_i(K|\,\cdot\,) \ast \delta$ for every $K \in \K^n$ of class $C^\infty_+$. The result for a general convex body $K \in \mathcal{K}^n$ follows by approximation (see, e.g., \cite{Schneider2014}*{Thm.~3.4.1}) and Lemma \ref{lem:convergenceConv}: Indeed, let $K \in \K^n$ be arbitrary and approximate $K$ by a sequence $K_j \in \K^n$ of bodies of class $C^\infty_+$. Then the projection functions $V_i(K_j|\cdot)$ converge uniformly to $V_i(K|\cdot)$. In particular, they converge weakly as distributions. By Lemma~\ref{lem:convergenceConv} (using that convolution on $\Gr_{i,n}$ is induced by convolution on $\SO(n)$), $V_i(K_j|\cdot) \ast \delta \to V_i(K|\cdot) \ast \delta$.
		
		On the other hand, by the continuity of $\Phi_i$, $V_i(K_j|\cdot) \ast \delta=h(\Phi_i K_j, \cdot) \to h(\Phi_i K, \cdot)$ uniformly, in particular also weakly. As limits are unique in the space of distributions, $h(\Phi_i K, \cdot)$ and $V_i(K|\cdot) \ast \delta$ must define the same distribution, that is, the distribution $V_i(K|\cdot) \ast \delta$ is given by the continuous function $h(\Phi_i K, \cdot)$.
		
		\medskip
		
		It remains to show the uniqueness of the spherical Crofton distribution $\delta = \widehat{\psi}$. To this end, we first note that by its $\SO(n-1)$ invariance, $\psi$ is uniquely determined by its values on $\SO(n-1)$ invariant functions in $C^{\infty}(\Gr_{i,n})$. Next, using Theorem~\ref{thm:AlFExtCrKl}, we see that for $f \in C^{\infty}(\Gr_{i,n})$,
		\begin{equation} \label{prfthm1a}
			\pair{\psi}{C_i f}_{C^{-\infty}} = \pair{\widetilde{\Kl}_i(\widetilde {\Cr_{i}} \psi)}{f }_{C^{-\infty}} =
			\pair{\widetilde{\Kl}_i(\widetilde{\pd}\, \varphi_i)}{f}_{C^{-\infty}}
		\end{equation}
		depends only on the valuation $\varphi_i$. If $f$ is in addition $\SO(n-1)$ invariant, then it was shown in \cite{Schuster2015}*{Lem.~3.3} that
		\begin{align*}
			C_i f = \frac{n\kappa_i\kappa_{n-i}}{2\kappa_{n-1}} {\binom{n}{i}}^{-1} \widehat{C \hat{f}},
		\end{align*}
		where $C$ denotes the spherical cosine transform (see, e.g., \cite{Groemer1996}*{Sec.~3.4}), and
		when $\hat{f}$ is interpreted as an $\mathrm{S}(\OO(i) \times \OO(n-i))$ invariant function in $C^{\infty}(\S^{n-1})$. Since such functions are even (see \cite{Schuster2015}*{Prop.~A}) and the spherical cosine transform $C$ is bijective on even functions in $C^{\infty}(\S^{n-1})$, $C_i$ is a bijection on $\SO(n-1)$ invariant functions in $C^{\infty}(\Gr_{i,n})$. Consequently, by \eqref{prfthm1a}, the values
		\begin{align*}
			\pair{\psi}{f}_{C^{-\infty}} =\pair{\widetilde{\Kl}_i(\widetilde{\pd}\, \varphi_i)}{C_i^{-1}f }_{C^{-\infty}}
		\end{align*}
		for $\SO(n-1)$ invariant $f \in C^{\infty}(\Gr_{i,n})$ are uniquely determined by $\varphi_i$.
	\end{proof}

	\section{Proof of \texorpdfstring{\cref{mthm:ineqEvenMinkVal}}{Theorem B}}
	In this section, we give a proof of \cref{mthm:ineqEvenMinkVal}. To this end, we first prove the left-hand inequality of \eqref{eq:mthmIneq}, which holds in larger generality. Next, we will prove some auxiliary result for $\SO(n-1)$ orbits in the Grassmanian, which is needed for the proof of the equality cases of \cref{mthm:ineqEvenMinkVal}, given then. In the remainder of the section, we will show in \cref{ex:NonZonoidPosCrofton} that \cref{mthm:ineqEvenMinkVal} significantly extends previous results by \cite{Berg2020}.
	
	\medskip
	
	First, let us state the following well-known direct consequence of \cref{thm:repGenFct}.
	\begin{lem}[see, e.g., \cite{Schuster2006}*{Lem.~6.3}]
		\label{lem:propNormMinkVal}
		Let $\Phi_i$ be an $i$-homogeneous Minkowski valuation, $1 \leq i \leq n-1$, generated by $f \in L^1(\S^{n-1})$. If $K \in \K^n$, then
		\begin{align*}
			\Phi_i B^n = r_{\Phi_i} B^n \qquad \text{and} \qquad W_{n-1}(\Phi_i K) = r_{\Phi_i} W_{n-i}(K),
		\end{align*}
		where $r_{\Phi_i} = \int_{\S^{n-1}} f(u) du$.
	\end{lem}
	
	We further need the following inequality due to Lutwak~\cite{Lutwak1975} (used for deducing Urysohn's inequality from the classical Blaschke--Santal\'o inequality), stating that if $K \in \K^n$ contains the origin in its interior, then
	\begin{align}\label{eq:PolMeanWidthIneqLutwak}
		V_n(K^\circ)(nW_{n-1}(K))^{n} \geq \kappa_n,
	\end{align}
	with equality if and only if $K$ is a centered Euclidean ball.
	
	The left-hand side of \eqref{eq:mthmIneq} reads
	\begin{prop}	\label{prop:ineqEvenMinkValQm}
		Suppose that $\Phi_i : \mathcal{K}^n \rightarrow \mathcal{K}^n$, $1 \leq i \leq n - 1$, is a continuous, translation invariant, and $\SO(n)$ equivariant Minkowski valuation, mapping bodies with non-empty interior to such bodies. If $K \in \K^n$ has non-empty interior, then
		\begin{align}\label{eq:MinkValEstQm}
			V_n(\Phi_i^\circ K)(nW_{n-i}(K))^n \geq V_n(\Phi_i^\circ B^n).
		\end{align}
		Equality holds if and only if $\Phi_i K$ is a Euclidean ball.
	\end{prop}
	Let us point out that the conditions of \cref{prop:ineqEvenMinkValQm} are trivially fulfilled, if $\Phi_i$ admits a non-negative spherical Crofton measure which is non-zero.
	\begin{proof}[Proof of \cref{prop:ineqEvenMinkValQm}]
		Let $K \in \K^n$ with non-empty interior. Then, by assumption, $\Phi_i K$ has non-empty interior, which contains the Steiner point of $\Phi_iK$. As the Steiner point of $\Phi_iK$ lies at the origin (see, e.g., \cites{Alesker2011, Parapatits2012, Berg2014}), $0 \in \mathrm{int}\, \Phi_i K$, and by \eqref{eq:PolMeanWidthIneqLutwak},
		\begin{align*}
			V_n(\Phi_i^\circ K)(nW_{n-1}(\Phi_i K))^{n} \geq \kappa_n.
		\end{align*}
		Thus, the inequality follows directly from \cref{lem:propNormMinkVal}. Next, note that as its Steiner point always lies at the origin, $\Phi_i K$ is a Euclidean ball if and only if it is a centered Euclidean ball. The equality cases therefore follow directly from the equality cases of \eqref{eq:PolMeanWidthIneqLutwak}.
	\end{proof}
	
	\begin{lem}\label{lem:OrbitsSOn1}
		Let $E \in \Gr_{i,n}$, $1 \leq i \leq n-2$, such that neither $\bar e \in E$ nor $E \subset \bar e^\perp$. Set
		\begin{align*}
			t_E = \max_{v \in \S^{i-1}(E)}\pair{v}{\bar e} \in (0,1).
		\end{align*}
		If $\eta \in \SO(n)$ is such that
		\begin{align*}
			\eta(H_{\bar{e}, t_E} \cap \S^{n-1}) \cap (H_{\bar{e}, t_E} \cap \S^{n-1}) \neq \emptyset,
		\end{align*}
		then there exists $G \in \Gr_{i,n}$ such that $E, G, \eta^{-1}G$ are in the same $\SO(n-1)$ orbit.
	\end{lem}
	\begin{proof}
		First, note that two subspaces $E, F \in \Gr_{i,n}$ are in the same $\SO(n-1)$ orbit if and only if $t_E = t_F$. Indeed, $E \mapsto t_E$ clearly is invariant under $\SO(n-1)$ and the converse claim is trivial for $t_E=t_F\in \{0,1\}$. For $t_E \in (0,1)$, observe that $t_E = \pair{v_E}{\bar e}$ for a unique unit vector $v_E \in E$, as $\dim (E \cap \bar{e}^\perp) = i-1$, and similarly for $F$. Moreover, $v_E | \bar{e}^\perp$ resp. $v_F | \bar{e}^\perp$ are orthogonal to $E \cap \bar{e}^\perp$ resp. $F \cap \bar{e}^\perp$ and have the same length $\sqrt{1-t_E^2}$. We can therefore find $\tau \in \SO(n-1)$ mapping $E \cap \bar{e}^\perp$ to $F \cap \bar{e}^\perp$ and $v_E | \bar{e}^\perp$ to $v_F | \bar{e}^\perp$. Clearly, this implies that $\tau E = F$, that is, $E$ and $F$ are in the same $\SO(n-1)$ orbit.
		
		Next, let $\eta \in \SO(n)$ as above and take $w_1 \in \eta(H_{\bar{e}, t_E} \cap \S^{n-1}) \cap (H_{\bar{e}, t_E} \cap \S^{n-1})$. As each of the spaces $(\eta \bar{e})^{\perp}$, $w_1^{\perp}$, and $\bar{e}^{\perp}$ is $(n - 1)$-dimensional, their intersection has at least dimension $n - 3$. Since $i \leq n - 2$, we have $i - 1 \leq n - 3$ and, thus, there exists an orthonormal system $\{w_2, \ldots, w_i\}$ in $(\eta \bar{e})^{\perp} \cap w_1^{\perp} \cap \bar{e}^{\perp}$.
		
		 Define $G = \mathrm{span}\,\{w_1, w_2, \ldots, w_i\}$. It follows that $v_G = w_1$ and, thus, by the first step, that $E$ and $G$ are in the same $\SO(n-1)$ orbit. Moreover, since $\eta^{-1} G = \mathrm{span}\,\{\eta^{-1} w_1, \eta^{-1} w_2, \ldots, \eta^{-1} w_i\}$, where
		$\eta^{-1} w_1 \in H_{\bar{e},t} \cap \mathbb{S}^{n-1}$ and $\eta^{-1} w_j \in \bar{e}^{\bot}$, for all $j \geq 2$, it follows that $v_{\eta^{-1}G} = \eta^{-1}w_1$. Thus, by the above, $\eta^{-1}G$ is in the same $\SO(n-1)$ orbit as $E$ and $G$.
	\end{proof}
	
	We are now ready to complete the proof of \cref{mthm:ineqEvenMinkVal} by showing the right-hand inequality of \eqref{eq:mthmIneq}. We will prove it without any assumption on the normalization. However, to exclude pathological cases, we will assume that the Minkowski valuation is non-trivial, that is, it is not of the form $\Phi_i(K) = \{0\}$ for all $K \in \K^n$. As the cases $i=1$ and $i=n-1$ were previously proved in \cite{Hofstaetter2021} and \cite{Haberl2019}, respectively, and the proofs of the equality cases need slightly different arguments in co-degree $1$, we will restrict ourselves to $2 \leq i \leq n-2$.
	\begin{thm}	\label{thm:ineqEvenMinkValAffQm}
		Suppose that the spherical Crofton distribution of a non-trivial even Minkow\-ski valuation $\Phi_i : \mathcal{K}^n \rightarrow \mathcal{K}^n$, $2 \leq i \leq n - 2$, is non-negative.
		
		If $K \in \K^n$ has non-empty interior, then
		\begin{align}\label{eq:PosGenEstAffQ}
			V_n(\Phi_i^\circ K)A_{n-i}(K)^{n} \leq V_n(\Phi_i^\circ B^n)\kappa_n^n. 
		\end{align}
		Equality holds if and only if $K$ is of constant $i$-brightness, that is, there exists $c \in \R$ such that $V_i(K|F) = c$ for all $F \in \Gr_{i,n}$.
	\end{thm}
	\begin{proof}
		Let $K \in \K^n$ with non-empty interior. By assumption, $\Phi_i K$ is given by
		\begin{align*}
			h(\Phi_i K, \cdot) = V_i(K|\cdot) \ast \delta,
		\end{align*}
		where $\delta$ is an $\OO(i) \times \OO(n-i)$ invariant, non-negative distribution $\S^{n-1}$, that is a positive measure. Denoting the ($\SO(n-1)$ invariant) measure $\widehat{\delta}$ on $\Gr_{i,n}$ by $\mu$, we thus have by \eqref{eq:convGrSph}
		\begin{align}\label{eq:prfIneqRepInt}
			h(\Phi_i K, u) = \int_{\Gr_{i,n}}V_i(K|\vartheta_u F) d\mu(F), \quad u \in \S^{n-1},
		\end{align}
		for some $\vartheta_u \in \SO(n)$ with $\vartheta_u \bar e = u$. As \eqref{eq:PosGenEstAffQ} is invariant under scaling of $\Phi$, we may assume that $\Phi_i B^n = \kappa_i B^n$, that is, that $\mu$ is a probability measure. Let us further point out that by \eqref{eq:prfIneqRepInt}, $\Phi_i$ is monotone and thus $\Phi_i K$ contains a small ball around the origin whenever $K$ has non-empty interior.
		
		By using polar coordinates and \eqref{eq:prfIneqRepInt},
		\begin{align*}
			V_n(\Phi_i^\circ K) = \frac{1}{n} \int_{\S^{n-1}} h(\Phi_i K, u)^{-n}du = \frac{1}{n}\int_{\S^{n-1}} \left(\int_{\Gr_{i,n}} V_i(K|\vartheta_u F)d\mu(F)\right)^{-n} du,
		\end{align*}
		which, by Jensen's inequality for the $\mu$-integral is estimated by
		\begin{align}\label{eq:prfIneqJensen}
			V_n(\Phi_i^\circ K) \leq \frac{1}{n}\int_{\S^{n-1}} \int_{\Gr_{i,n}} V_i(K|\vartheta_u F)^{-n}d\mu(F) du.
		\end{align}
		Next, as both $\Phi_i$ and polarity intertwine $\SO(n)$ maps, $V_n(\Phi_i^\circ K) = V_n(\Phi_i^\circ (\tau K))$ for all $\tau \in \SO(n)$. Integrating with respect to the Haar probability measure on $\SO(n)$, denoted by $d\tau$, therefore yields
		\begin{align*}
			V_n(\Phi_i^\circ K) = \int_{\SO(n)}\!\!\!\!\!V_n(\Phi_i^\circ (\tau K))d\tau \leq \frac{1}{n}\int_{\SO(n)}\int_{\S^{n-1}} \int_{\Gr_{i,n}} \!\!\!\!\!\!\!\! V_i(K|\tau^{-1}\vartheta_u F)^{-n}d\mu(F) dud\tau.
		\end{align*}
		By Fubini's theorem and since the Haar measure on $\SO(n)$ is invariant, we can rewrite the right-hand side to
		\begin{align*}
			\frac{1}{n}\int_{\S^{n-1}} \int_{\Gr_{i,n}} \int_{\SO(n)} \!\!\! V_i(K|\eta F)^{-n}d\eta d\mu(F) du = \kappa_n \int_{\Gr_{i,n}}V_i(K|E)^{-n} dE,
		\end{align*}
		where we used that $\mu(\Gr_{i,n}) = 1$ and that by the uniqueness of the Haar measure
		\begin{align*}
			\int_{\SO(n)} f(\eta F) d\eta = \int_{\Gr_{i,n}} f(E) dE, \qquad \forall f \in C(\Gr_{i,n}), F \in \Gr_{i,n}.
		\end{align*}
		Identifying $A_{n-i}(K)$, we thus have shown the claimed inequality
		\begin{align*}
			V_n(\Phi_i^\circ K) \leq \frac{\kappa_n^{n+1}}{\kappa_i^n} A_{n-i}(K)^{-n} = \kappa_n^{n}V_n(\Phi_i^\circ K)A_{n-i}(K)^{-n}.
		\end{align*}
		It remains to prove the equality cases. To this end, we note that by the first part of the proof, equality holds in \eqref{eq:PosGenEstAffQ} if and only if equality holds in every instance of \eqref{eq:prfIneqJensen}. By the equality conditions of Jensen's inequality, this is the case if and only if for almost every $\tau \in \SO(n)$ and $u \in \S^{n-1}$, there exist $c_{\tau, u} \in \R$ such that 
		\begin{align}\label{eq:IneqCondEq1}
			V_i(K|\tau^{-1}\vartheta_u F) = c_{\tau, u}, \qquad \text{ for $\mu$-a.e. } F \in \Gr_{i,n}.
		\end{align}
		Since $\eta \mapsto V_i(K|\eta F)$ is continuous on $\SO(n)$, \eqref{eq:IneqCondEq1} holds indeed for every $\tau \in \SO(n)$ and $u \in \S^{n-1}$, and, by the continuity of $V_i(K|\cdot)$, for all $F \in \mathrm{supp}\,\mu$. Clearly, this is equivalent to the existence of $c_\eta \in \R$, $\eta \in \SO(n)$, such that
		\begin{align}\label{eq:IneqCondEq2}
			V_i(K|\eta F) = c_\eta, \qquad \text{ for all } F \in \mathrm{supp}\,\mu.
		\end{align}
		Since $\mu$ is non-zero, there exists a subspace $E \in \mathrm{supp}\,\mu$ and, by the invariance of $\mu$, also its $\SO(n-1)$ orbit is contained in $\mathrm{supp}\,\mu$. We distinguish now three cases:
		\begin{enumerate}
			\item \emph{There exists $E \in \mathrm{supp}\,\mu$ such that $\bar{e} \in E$.}\\
			Then the $\SO(n-1)$ orbit of $E$ consists of all $i$-dimensional subspaces containing $\bar{e}$, and \eqref{eq:IneqCondEq2} implies that
			\begin{align*}
				V_i(K|\eta F) = c_\eta, \qquad \text{ for all } F \in \Gr_{i,n} \text{ with } \bar e \in F.
			\end{align*}
			Consequently, for every $u \in \S^{n-1}$, $V_i(K|G) = c_{\vartheta_u}$ for all $G \in \Gr_{i,n}$ containing $u$. Since $i \geq 2$, for every $u,v \in \S^{n-1}$, we find $G \in \Gr_{i,n}$ containing both of them. Hence, $c_{\vartheta_u} = c_{\vartheta_v}$, that is, $u \mapsto c_{\vartheta_u}$ is independent of $u \in \S^{n-1}$, concluding the proof in this case.
			\smallskip
			\item \emph{There exists $E \in \mathrm{supp}\,\mu$ such that $E \subset \bar{e}^\perp$.}\\
			Then the $\SO(n-1)$ orbit of $E$ consists of all $i$-dimensional subspaces that are contained in $\bar{e}^\perp$, and we infer as in the previous case that
			\begin{align*}
				V_i(K|G) = c_{\vartheta_u}, \qquad \text{ for all } G \in \Gr_{i,n} \text{ such that } G \subset u^\perp.
			\end{align*}
			Since $i \leq n-2$, every such $G\in \Gr_{i,n}$ is contained in at least two hyperplanes $u^\perp$, $v^\perp$, and we can conclude as in (i) that $V_i(K|\cdot)$ is constant.
			\smallskip
			\item \emph{For all $E \in \mathrm{supp}\,\mu$ neither $\bar e \in E$ nor $E \subset \bar{e}^\perp$.}\\
			Let $t = \max_{v \in \S^{i-1}(E)}\pair{v}{\bar{e}} \in (0,1)$ and note that the set of all $\eta \in \SO(n)$ such that
			\begin{align*}
				\eta(H_{\bar{e}, t} \cap \S^{n-1}) \cap (H_{\bar{e}, t} \cap \S^{n-1}) \neq \emptyset
			\end{align*}
			contains an open neighborhood of the identity in $\SO(n)$. For a fixed such $\eta$, by \cref{lem:OrbitsSOn1}, there exists a subspace $G \in \Gr_{i,n}$ with both $G$ and $\eta^{-1}G$ in the $\SO(n-1)$ orbit of $E$, and, thus, in the support of $\mu$. Hence, we have on one hand, by \eqref{eq:IneqCondEq2}, that $c_\tau = V_i(K|\tau G)$ for every $\tau \in \SO(n)$. On the other hand, again by \eqref{eq:IneqCondEq2}, we have $V_i(K|\tau G)= V_i(K|\tau\eta (\eta^{-1}G)) = c_{\tau\eta}$. Consequently,
			\begin{align*}
				c_\tau = V_i(K|\tau G) = c_{\tau\eta},
			\end{align*}
			that is, the continuous map $\eta \mapsto c_\eta$ is locally constant on $\SO(n)$ and therefore constant as $\SO(n)$ is connected. This concludes the proof.
		\end{enumerate}
	\end{proof}

	Finally, we will give examples of Minkowski valuations of degree $2 \leq i \leq n-2$, which are eligible for our \cref{mthm:ineqEvenMinkVal}, but not for the preceding theorem in \cite{Berg2020}. In \cite{Berg2020}, inequality~\eqref{eq:mthmIneq} is shown for all Minkowski valuations $\Phi_i$, such that
	\begin{align*}
		h(\Phi_i K, \cdot) = S_i(K, \cdot) \ast h(Z, \cdot), \quad K \in \K^n,
	\end{align*}
	where $Z \in \K^n$ is a zonoid of revolution, that is, a body of revolution which is the limit of finite Minkowski sums of segments. Alternatively, $Z$ is a zonoid of revolution, if and only if there exists a non-negative, $\SO(n-1)$ invariant measure $\rho\in\mathcal{M}(\S^{n-1})$ such that (up to translations)
	\begin{align*}
		h(Z, u) = \int_{\S^{n-1}} |\pair{u}{v}| d\rho(v) = (C\rho)(u), \quad u \in \S^{n-1},
	\end{align*}
	where $C$ denotes the spherical cosine transform, see, e.g., \cite{Schneider2014}*{Thm.~3.5.3}. A body is called a generalized zonoid, if its support function up to translations is given as the spherical cosine transform of a signed measure.
	
	In order to compare \cref{mthm:ineqEvenMinkVal} and the result from \cite{Berg2020}, we will need the following theorem translating between the representations from \cref{mthm:repEvenMinkVal} and \cref{thm:repGenFct}. Here, $R_{i,j}:C(\Gr_{i,n}) \to C(\Gr_{j,n})$, $i\neq j \in \{1, \dots, n-1\}$, denotes the Radon transform, defined by
	\begin{align*}
		(R_{i,j}f)(F) = \int_{\Gr_{i,n}^F} f(E) dE, \quad F \in \Gr_{j,n},
	\end{align*}
	where $\Gr_{i,n}^F$ denotes the subset of $\Gr_{i,n}$ of all subspaces that contain resp. are contained in $F$, depending on whether $i > j$ or $i < j$. The transform is extended to measures and distributions using that the adjoint of $R_{i,j}$ is given by $R_{j,i}$.
		
	\begin{thm}[\cite{Schuster2015}*{Cor.~5.1}]
		\label{coro:translationGenMeasGenFct}
		Suppose that $\Phi_i$ is an even $i$-homogeneous Min\-kowski valuation such that $h(\Phi_i K, \cdot) = S_i(K, \cdot) \ast (C \nu)$, where $\nu$ is a signed measure on $\S^{n-1}$.
		Then the spherical Crofton measure $\mu$ of $\Phi_i$ is given by
		\begin{align*}
			\hat \mu = 2 \frac{\kappa_{n-1}}{\kappa_{i}} (R_{n-1,i} \nu^\perp).
		\end{align*}
	\end{thm}
	Clearly, \cref{coro:translationGenMeasGenFct} implies that if $\Phi_i$ is generated by the support function of a zonoid, then it admits a non-negative spherical Crofton distribution. In the next lemma, we use \cref{coro:translationGenMeasGenFct} to give a condition on a generating function of the form $Cg$ for some even $\SO(n-1)$ invariant function $g \in C(\S^{n-1})$. We will make use of cylinder coordinates, see, e.g., \cite{Groemer1996}*{Lem.~1.3.1},
	\begin{align}\label{eq:cylCoord}
		\int_{\S^{n-1}} \bar{f}(\pair{w}{u})du
		= \omega_{n-1} \int_{-1}^{1} \bar{f}(t) (1-t^2)^{\frac{n-3}{2}} dt,
	\end{align}
	for all $w \in \S^{n-1}$, where we define $\omega_{\alpha}:=2\pi^{\frac{\alpha}{2}}/\Gamma(\frac{\alpha}{2})$ for $\alpha>0$.
	\begin{lem}
		\label{lem:condMinkValPosGen}
		Suppose that $\Phi_i:\K^n \to \K^n$ is an even $i$-homogeneous Min\-kowski valuation such that $h(\Phi_i K, \cdot) = S_i(K, \cdot) \ast (C g)$, where $g=\bar{g}(\pair{\cdot}{\bar e})\in C(\S^{n-1})$ is even. Then the spherical Crofton measure of $\Phi_i$ is non-negative, if and only if
		\begin{align*}
			\int_0^1 \bar g (\alpha t) (1-t^2)^{\frac{n-i-3}{2}} d t \geq 0 \quad \forall \alpha \in [0,1].
		\end{align*}
	\end{lem}
	\begin{proof}
		By Theorem~\ref{coro:translationGenMeasGenFct}, the Crofton measure of $\Phi_i$ is positively generated, if and only if
		\begin{align}\label{eq:prfPosGenCond1}
			(R_{n-1,i} g^\perp)(F^\perp) = (R_{1,n-i} g)(F) = \int_{\S^{n-1} \cap F} g(w) d w \geq 0,
		\end{align}
		for every $F \in \Gr_{n-i,n}$. Noting that $\pair{w}{\bar e} = \pair{w}{\bar e | F}$ for every $w \in F$ and $\bar e|F \neq 0$ whenever $F \not \subset \bar e^\perp$, cylinder coordinates~\eqref{eq:cylCoord} yield
		\begin{align*}
			\int_{\S^{n-1} \cap F} g(w) d w = \omega_{n-i-1} \int_{-1}^{1} \bar{g}(t\|\bar e|F\|) (1-t^2)^{\frac{n-i-3}{2}} dt
		\end{align*}
		for almost all $F \in \Gr_{n-i,n}$. By continuity this holds for all $F$.
		
		As we vary $F \in \Gr_{n-i,n}$, $\|\bar e|F\|$ ranges from $0$ to $1$. Consequently, condition~\eqref{eq:prfPosGenCond1} is -- omitting the positive constants -- equivalent to
		\begin{align*}
			\int_{-1}^1 \tilde g (\alpha t) (1-t^2)^{\frac{n-i-3}{2}} d t \geq 0 \quad \forall \alpha \in [0,1].
		\end{align*}
		Since $g$ is assumed to be even, the claim follows.
	\end{proof}

	\medskip

	We are now ready to give the examples. It uses a generalization of an example from \cite{Gardner2006}*{Rem.~4.1.14} and \cite{Schneider1970}*{p.~69}, given in \cite{OrtegaMoreno2021} (see also \cite{Dorrek2017}*{Ex. 5.2(f)}). 
	
	\begin{exl}\label{ex:NonZonoidPosCrofton}
		Let $P_2^n(t) = \frac{1}{n-1}(nt^2 - 1)$ be the second Legendre polynomial and consider for $\alpha \in \R$ the function $h_\alpha(u)= 1 + \alpha P_2^n(\pair{u}{\bar e})$, $u \in \S^{n-1}$. By \cite{OrtegaMoreno2021}*{Prop.~5.4}, $h_\alpha \in C(\S^{n-1})$ is the support function of a convex body of revolution $L_\alpha$ if and only if
		\begin{align*}
			\alpha \in \left[-\frac{n-1}{2n-1}, \frac{n-1}{n+1}\right].
		\end{align*}
		For $\alpha$ in the given interval, we thus get a one-parametric family of convex bodies, where the support function is a sum of two (even) spherical harmonics. Hence, $h_\alpha$ lies in the image of the cosine transform, that is,
		\begin{align*}
			h_\alpha(u) = 1 + \alpha P_2^n(\langle u, \bar e \rangle) = C \left(\frac{1}{2 \kappa_{n-1}} (1 + (n+1) \alpha P_2^n(\langle \cdot, \bar e \rangle)\right),
		\end{align*}
		where we used the multipliers of the cosine transform (see, e.g., \cite{Groemer1996}*{Lemma~3.4.5}).
		
		Whenever $h_\alpha$ is a support function, the resulting convex body is therefore (by definition) a generalized zonoid. It is a zonoid, if the preimage of $h_\alpha$ under the cosine transform is non-negative, that is, by a direct calculation, exactly for $\alpha \geq -\frac{1}{n+1}$.
		
		By Lemma~\ref{lem:condMinkValPosGen}, the (generalized) zonoid with support function $h_\alpha$ generates a Minkowski valuation of degree $i$ with non-negative Crofton distribution, if and only if
		\begin{align*}
			\int_0^1 \left(1 + (n+1) \alpha P_2^n(\tau t) \right)(1-t^2)^{\frac{n-i-3}{2}} d t \geq 0 \quad \forall \tau \in [0,1].
		\end{align*}
		Calculating the integral using the Beta function yields that this condition is satisfied exactly for $\frac{n+1}{n-1}\alpha \in \left[ -\frac{n-i}{i}, 1 \right]$.
		
		Finally, intersecting the intervals we obtain that for all 
		\begin{align*}
			\alpha \in \left[\max\left\{- \frac{n-i}{i}\frac{n-1}{n+1}, \frac{n-1}{1-2n} \right\}, -\frac{1}{n+1}\right),
		\end{align*}
		$L_\alpha$ is a generalized zonoid (and not a zonoid) and the Minkowski valuation defined by $S_i(K,\cdot) \ast h_\alpha$ has a non-negative Crofton distribution. Let us point out that for a fixed $i$, we asymptotically get the interval $\left[ -\frac{n-1}{2n-1}, -\frac{1}{n+1}\right)$, while for fixed $n$ and $i<n-1$ the interval always has positive length, yielding an infinite family of examples.	\cref{mthm:ineqEvenMinkVal} therefore significantly extends the results from \cite{Berg2020}.
	\end{exl}

	\section{Existence of Extremals}\label{sec:exExtr}
	In this section, we first give a proof of \cref{mthm:existExtr}. Then we give an example of an even Minkowski valuation of degree $1$, which is not monotone, but satisfies the assumptions of \cref{mthm:existExtr}. This shows that (weak) monotonicity is not a necessary conditions for the existence of extremals. In the remainder of the section, we give a proof of the analogous problem for non-polar inequalities in the setting of convex bodies of revolution.
	
	First, let us recall the well-known isoperimetric inequality relating two different quermassintegrals (see, e.g., \cite{Schneider2014}*{Ch.~7}),
	\begin{align}
		\label{eq:IsoperimetricIneqTwoQuermass}
		\kappa_n^{-i} W_{n-j}(K)^{i} \geq \kappa_n^{-j} W_{n-i}(K)^{j},
	\end{align}
	where $0 < j < i \leq n$, $K \in \K^n$ with non-empty interior. Equality holds in \eqref{eq:IsoperimetricIneqTwoQuermass} if and only if $K$ is a Euclidean ball. Setting one parameter to be $n$, we obtain the isoperimetric inequality between volume and $W_{n-i}$, $1 \leq i \leq n - 1$. Letting $S(K)=nW_1(K)$ be the surface area of $K$, this reduces to the classical isoperimetric inequality, see also \eqref{WiVninequ},
	\begin{equation} \label{eq:isoperIneq}
		W_{n-i}(K)^n \geq \kappa_n^{n-i}V_n(K)^i \qquad \text{ and} \qquad \frac{S(K)^n}{(n\kappa_n)^n} \geq \frac{V_n(K)^{n-1}}{\kappa_n^{n-1}}.
	\end{equation}
	Equality holds if and only if $K$ is a Euclidean ball.
	
	In order to prove the essential \cref{lem:isoperRatioTendsInfty} in the proof of \cref{mthm:existExtr}, we need the following result from \cite{Gritzmann1987} (see also \cite{Gritzmann1988} or \cite{Esposito2005}*{Lem~4.1}), stating that
	\begin{align}\label{ineqDiamGritzmann}
		\mathrm{diam}\,  K \leq c(n) \frac{S(K)^{n-1}}{V_n(K)^{n-2}},
	\end{align}
	for $K \in \K^n$ with non-empty interior and some constant $c(n) > 0$ depending only on the dimension.
	
	\begin{lem}
		\label{lem:isoperRatioTendsInfty}
		Suppose that $(K_j)_{j \in \N}$ is a sequence of convex bodies with non-empty interior. If $(K_j)_{j \in \N}$ converges to $K_0 \in \K^n$ of dimension $1 \leq \dim K_0 \leq n-1$, then the $i$-th isoperimetric ratio, $0 \leq i \leq n-1$, tends to infinity,
		\begin{align*}
			\frac{W_{n-i}(K_j)^n}{V_n(K_j)^{i}} \to \infty, \qquad j \to \infty.
		\end{align*}
	\end{lem}
	\begin{proof}
		First, let $i=n-1$. By \eqref{ineqDiamGritzmann} and the isoperimetric inequality~\eqref{eq:isoperIneq},
		\begin{align}\label{eq:prfIsoperRatInf}
			\frac{S(K_j)^n}{V_n(K_j)^{n-1}} = \frac{S(K_j)}{V_n(K_j)}\frac{S(K_j)^{n-1}}{V_n(K_j)^{n-2}} \geq \frac{S(K_j)}{V_n(K_j)} \frac{\mathrm{diam}\,  K_j}{c(n)} \geq \frac{n\kappa_n^{\frac{1}{n}}}{c(n)} \frac{\mathrm{diam}\,  K_j}{V_n(K_j)^{\frac{1}{n}}}.
		\end{align}
		By the assumptions on $\dim K_0$, $\mathrm{diam}\,  K_j \to \mathrm{diam}\,  K_0 \neq 0$ and $V_n(K_j) \to V_n(K_0) = 0$. Hence, the right-hand side of \eqref{eq:prfIsoperRatInf} tends to infinity, concluding the proof for $i=n-1$ as $nW_1(K) = S(K)$ for all $K \in \K^n$.
		
		For general $0 \leq i < n-1$, by \eqref{eq:IsoperimetricIneqTwoQuermass},
		\begin{align*}
			\kappa_n^{-n+1} W_{n-i}(K)^{n-1} \geq \kappa_n^{-i} W_{1}(K)^i,
		\end{align*}
		and, thus,
		\begin{align*}
			\frac{W_{n-i}(K_j)^n}{V_n(K_j)^{i}} \geq \kappa_n^{\frac{(n-i-1)n}{n-1}} \left(\frac{W_1(K_j)^n}{V_n(K_j)^{n-1}}\right)^{\frac{i}{n-1}}.
		\end{align*}
		Consequently, the claim for general $i$ follows from the case $i=n-1$.
	\end{proof}

	To show the existence of Minkowski valuations for which the volume product admits maximizers different from balls, we need the following example from \cite{Hofstaetter2021}.
	\begin{thm}[\cite{Hofstaetter2021}*{Thm.~4.1}]\label{thm:Junbounded}
		For every $n \geq 2$, the volume product
		\begin{align*}
			V_n(\mathrm{J}^\circ K)V_n(K), \quad K \in \K^n,
		\end{align*}
		is unbounded for $\mathrm{J}: \K^n \to \K^n, \mathrm{J}K = K - s(K)$, where $s(K)$ denotes the Steiner point of $K$.
	\end{thm}
	
	We are now ready to prove \cref{mthm:existExtr}.
	\begin{proof}[Proof of \cref{mthm:existExtr}]
		Let $f \in L^1(\mathrm{S}^{n-1})$ denote the generating function of $\Phi_i$ and assume that $f = f_1 + f_2$, where $f_1$ and $f_2$ are generating functions of Minkowski valuations $\Phi_i^1$ and $\Phi_i^2$, respectively, and $f_1 > c > 0$ for some $c \in \mathbb{R}$.
		
		By the additivity of convolution and the fact that the Steiner point of $\Phi_i^2$ is at the origin, see, e.g., \cites{Alesker2011, Parapatits2012, Berg2014}, and is contained in $\mathrm{relint}\, \Phi_i^2K$,
		\begin{align*}
			\Phi_i K = \Phi_i^1 K + \Phi_i^2 K \supseteq \Phi_i^1 K.
		\end{align*}
		Moreover, for $u \in \S^{n-1}$ and $K \in \K^n$, a direct estimate shows that
		\begin{align*}
			h(\Phi_i^1 K, u) = \int_{\S^{n-1}} f_1(\vartheta_u^{-1}v) dS_i(K, v) \geq c n W_{n-i}(K),
		\end{align*}
		that is, $\Phi_i^1 K \supseteq c n W_{n-i}(K) B^n$. Hence, by the isoperimetric inequality~\eqref{eq:isoperIneq},
		\begin{align}\label{eq:prfExistEst}
			V_n(\Phi_i^{\circ}K)V_n(K)^i \leq V_n(\Phi_i^{1,\circ}K)V_n(K)^i \stackrel{(\ast)}{\leq} d_n \frac{V_n(K)^i}{W_{n-i}(K)^n} \leq d_n \kappa_n^{i-n} < \infty,
		\end{align}
		for every $K \in \K^n$ with non-empty interior, and where $d_n >0$ depends on $n$ only.
		
		Consequently, $V_n(\Phi_i^{\circ}K)V_n(K)^i$ is bounded from above and, therefore, we can choose a sequence $K_j \in \K^n$ of convex bodies with non-empty interior such that
		\begin{equation} \label{proofexist2}
			V_n(\Phi_i^{\circ}K_j)V_n(K_j)^i \to \sup_{\substack{K \in \mathcal{K}^n\\ \mathrm{int}\,K \neq \emptyset}} V_n(\Phi_i^{\circ}K)V_n(K)^i < \infty \quad \mbox{as $j \to \infty$.}
		\end{equation}
		Note that this supremum is positive, as can easily be confirmed by plugging in Euclidean balls. Since $V_n(\Phi_i^{\circ}K_j)V_n(K_j)^i$ is invariant under translations and scaling, we may assume that every $K_j$ is contained in $B^n$ and contains a segment of length one. Using Blaschke's selection theorem, we obtain a subsequence, which we again denote by $K_j$, converging to a convex body $K_0 \in \K^n$ with $\mathrm{dim}\,K_0 \geq 1$.
		
		Assume that $\mathrm{dim}\,K_0 < n$. Then, by $(\ast)$ in \eqref{eq:prfExistEst} and \cref{lem:isoperRatioTendsInfty}, we have
		\[V_n(\Phi_i^{\circ}K_j)V_n(K_j)^i \leq d_n \frac{V_n(K_j)^i}{W_{n-i}(K_j)^n} \to 0  \quad \mbox{as $j \to \infty$}, \]
		which is a contradiction to the convergence of $V_n(\Phi_i^{\circ}K_j)V_n(K_j)^i$ to the positive supremum~\eqref{proofexist2}. Hence, $K_0$ has non-empty interior and, by continuity, is a maximizer of $V_n(\Phi_i^{\circ}K)V_n(K)^i$.
		
		It remains to construct a Minkowski valuation $\Psi_1: \mathcal{K}^n \rightarrow \mathcal{K}^n$ of degree 1 such that the maximizers of $V_n(\Psi_1^{\circ}K)V_n(K)$ are different from Euclidean balls.
		To this end, let $\Phi_1: \mathcal{K}^n \rightarrow \mathcal{K}^n$ be a Minkowski valuation whose generating function $f$ is bounded away from zero; take, e.g., $f$ to be the support function of a convex body of revolution with non-empty interior or as in the proof of \cref{thm:monotExMax}. For $0 < \lambda < 1$, we define
		\begin{align*}
			\Psi_1^{\lambda}: \mathcal{K}^n \rightarrow \mathcal{K}^n, \quad \Psi_1^{\lambda} = \lambda\,\Phi_1 + (1-\lambda)\mathrm{J}.  		
		\end{align*}
		Then, $\Psi_1^{\lambda}$ satisfies the assumptions of the first part of the theorem and, therefore, $V_n(\Psi_1^{\lambda,\circ}K)V_n(K)$ admits maximizers for every $0 < \lambda < 1$. Moreover, since by \cref{lem:propNormMinkVal},
		\begin{align*}
			\Psi_1^{\lambda} B^n = (\lambda r_{\Phi_1} + 1 - \lambda) B^n \supseteq \min\{r_{\Phi_1},1\}B^n,  
		\end{align*}
		we see that $V_n(\Psi_1^{\lambda,\circ}B^n)V_n(B^n)$ is bounded by $\kappa_n^2\min\{r_{\Phi_1},1\}^{-n}$ for all $\lambda \in (0,1)$.
		Using that $\Psi_1^{\lambda} K \to \mathrm{J}K$ for every $K \in \K^n$ as $\lambda \to 0$, it follows from continuity and \cref{thm:Junbounded} that there exists a convex body $K \in \K^n$ with non-empty interior such that
		\begin{align*}
			V_n(\Psi_1^{\lambda,\circ}K)V_n(K) > \frac{\kappa_n^2}{\min\{r_{\Phi_1},1\}^n} \geq V_n(\Psi_1^{\lambda,\circ}B^n)V_n(B^n),
		\end{align*}
		whenever $\lambda > 0$ is small enough. Consequently, Euclidean balls cannot maximize $V_n(\Psi_1^{\lambda,\circ}K)V_n(K)$.
	\end{proof}

	\begin{rem}
		Let us point out that the proof of the second part of \cref{mthm:existExtr} does not use the specific form of $\mathrm{J}$, but rather that its volume product is unbounded. Indeed, the same arguments can be applied to any example of an $i$-homogeneous Minkowski valuation for which the quantity from \cref{mthm:existExtr} is unbounded. So far, $\mathrm{J}$, $-\mathrm{J}$ and their multiples are the only known examples. It is an interesting problem whether such examples exist for $2 \leq i \leq n-1$, and whether such examples exist for $i=1$ which are even.
		
		Moreover, the arguments of the proof show that the set of Minkowski valuations with unbounded volume product is neither open nor convex (take, e.g., $\frac{1}{2}J + \frac{1}{2}(-J)$, which admits maximizers).
	\end{rem}

	We turn now to Minkowski valuations of degree $1$. In this setting, it was proven in \cite{Hofstaetter2021} that the volume product $V_n(\Phi^\circ K)V_n(K)$ is maximized by Euclidean balls for all continuous, translation invariant and $\SO(n)$ equivariant Minkowski valuations which are montone with respect to set-inclusion. However, monotonicity is not a requirement for the existence of extremals, as the following theorem shows.
	\begin{thm} \label{thm:monotExMax}
		There exists an even continuous Minkowski valuation $\Phi_1: \K^n \rightarrow \K^n$ of degree $1$ which is translation invariant, $\SO(n)$ equivariant, and \emph{not monotone} such that
		\begin{align*}
			V_n(\Phi_1^\circ K)V_n(K)
		\end{align*}
		attains a maximum on convex bodies $K \in \K^n$ with non-empty interior.
	\end{thm}

	For the proof of \cref{thm:monotExMax}, we will need the following characterization of $1$-homogeneous Minkowski valuations from \cite{Kiderlen2006}. Here, a map $\Phi: \K^n \to \K^n$ is called weakly monotone, if it is monotone on the subset of convex bodies with Steiner point at the origin.
	\begin{thm}[\cite{Kiderlen2006}]\label{thm:kiderlenMonMinkEnd}
		A continuous, translation invariant, and $\SO(n)$ equivariant Minkowski valuation $\Phi_1:\K^n\to\K^n$ of degree $1$ is weakly monotone, if and only if 
		\begin{align*}
			h(\Phi_1 K, \cdot) = h(K, \cdot) \ast \mu, \quad K \in \K^n,
		\end{align*}
		for a centered, $\SO(n-1)$-invariant measure $\mu$ on $\S^{n-1}$, which is non-negative up to addition of a linear measure $\alpha \pair{u}{\bar e}du$, $\alpha \in \R$. It is monotone, if and only if $\mu$ is non-negative. Moreover, the measure $\mu$ is uniquely determined by $\Phi_1$.
	\end{thm}
	As \cref{mthm:existExtr} is formulated in terms of generating functions, we need to translate between this representation and the representation from \cref{thm:kiderlenMonMinkEnd}. To this end, let $\Box_n$ be the differential operator given by
	\begin{align*}
		\Box_n h = h + \frac{1}{n-1} \Delta_{\S} h, \quad h \in C^\infty(\S^{n-1}),
	\end{align*}
	where $\Delta_{\S}$ denotes the spherical laplacian. It satisfies
	\begin{align}\label{eq:boxSuppFct}
		\Box_n h(K, \cdot) = S_1(K, \cdot), \quad K \in \K^n,
	\end{align}
	where the inequality has to be understood in a distributional sense. In particular, using its self-adjointness, we obtain for all $K \in \K^n$ and $f \in C^\infty(\S^{n-1})$ that
	\begin{align}\label{eq:boxConv}
		S_1(K, \cdot) \ast f = h(K, \cdot) \ast \Box_n f.
	\end{align}
	Consequently, we will need to apply the inverse transform of $\Box_n$, which was described in \cite{Berg1969} by
	\begin{align}\label{eq:invBox}
		\Box_n^{-1}h (u) = \int_{\S^{n-1}} g_n(\pair{u}{v}) h(v) d v, \quad u \in \S^{n-1}, h \in C^{\infty}(\S^{n-1}),
	\end{align}
	where $g_n \in C^\infty(-1,1)$ denotes the $n$th Berg function, defined in \cite{Berg1969}.
	
	\medskip
	
	We are now ready to prove \cref{thm:monotExMax}, based on a construction from \cite{Dorrek2017b}.
	\begin{proof}[Proof of \cref{thm:monotExMax}]
		Since, by \eqref{eq:boxConv},
		\begin{align*}
			S_1(K, \cdot) \ast f = h(K, \cdot) \ast \Box_n f,
		\end{align*}
		by \cref{mthm:existExtr} and \cref{thm:kiderlenMonMinkEnd}, it is sufficient to find a strictly positive generating function $f \in C^\infty(\S^{n-1})$ of a $1$-homogeneous Minkowski valuation $\Phi_1$ for which $\Box_n f$ takes negative values. In order to find such a generating function, we will make use of a construction by Dorrek~\cite{Dorrek2017b} of a non-monotone even Minkowski valuation of degree $1$.
		
		Indeed, denote by $C_\alpha = \{u \in \S^{n-1}: \, \pair{u}{\bar e} \geq 1-\alpha\}$ the spherical cap around the pole $\bar e$ of radius $\alpha > 0$. Then Dorrek showed the following (combining his proofs of \cite{Dorrek2017b}*{Lem.~3.4 and Thm.~3.5}): For every even, non-negative and $\SO(n-1)$ invariant $h \in C(\S^{n-1})$ which is supported in $C_\alpha \cup (-C_\alpha)$ for sufficiently small $\alpha > 0$ and attains its maximum at $\bar e$, the equation $h(\Phi_1 K, \cdot) = h(K, \cdot) \ast (1-h)$ defines a Minkowski valuation.
		
		Choosing $h(\bar e) = C > 1$, these Minkowski valuations will not be monotone (and also not weakly monotone since they are even), by \cref{thm:kiderlenMonMinkEnd}, since $(1-h)$ takes negative values. It therefore remains to show that we can choose $h \in C^\infty(\S^{n-1})$ within the above restrictions such that $f = \Box_n^{-1} (1-h)$ is strictly positive.
		
		To this end, note that since $\Box_n 1 = 1$, by \eqref{eq:invBox}, this is equivalent to finding $h$ such that
		\begin{align}\label{eq:prfExMaxNonMon}
			\Box_n^{-1}h(u) = \int_{\S^{n-1}} g_n(\pair{u}{v})h(v) dv < 1, \quad \forall u \in \S^{n-1}.
		\end{align}
		We will prove that \eqref{eq:prfExMaxNonMon} holds for all $h$ as above, such that the radius $\alpha$ of the cap $C_\alpha$ containing the support is small enough. Therefore, let $\alpha_j > 0$ be a monotone sequence converging to $0$ and let $h_j \in C^\infty(\S^{n-1})$, $j \in \N$, be even with $\mathrm{supp}\, h_j \subseteq C_\alpha \cap (-C_\alpha)$ and $\max_{v \in \S^{n-1}} h(v) = h(\bar{e}) = C > 1$ for some fixed $C$.
		
		Next, fix $0<t<1$ arbitrary and let $A_t = \{u \in \S^{n-1}:\, |\pair{u}{e_n}| \leq t\}$. Then there exists $j_0 \in \N$ such that for all $j \geq j_0$, the closed spherical cap of radius $\alpha_j$ around any $u \in A_t$ does not contain $\bar e$. Fix $j \geq j_0$. By \cite{Berg1969}*{Thm.~3.3(iii)}, $|g_n(\pair{\cdot}{\bar e})|$ is continuous on $\S^{n-1} \setminus \{\bar{e}\}$ and, thus, bounded on the compact set $B_{t,j} := \{u \in \S^{n-1}: \pair{u}{\bar e} \leq 0 \text{ or } \mathrm{dist}(u, A_t) \leq \alpha_j\}$ by some constant $M>0$.
		
		 As for any $\vartheta_u \in \SO(n)$ with $\vartheta_u \bar e = u$, $\vartheta_u^{-1} v$ is contained in the spherical caps of radius $\alpha_j$ around $\pm u$ whenever $v \in C_{\alpha_j} \cup (-C_{\alpha_j})$. For $u \in A_t$, these caps are contained in $B_{t,j}$, and by $\pair{u}{v} = \pair{\vartheta_u^{-1}v}{\bar e}$, we deduce that $|g_n(\pair{u}{v})| \leq M$ for all $u \in A_t$ and $v\in C_{\alpha_j} \cup (-C_{\alpha_j})$. Hence, 
		\begin{align*}
			|\Box_n^{-1} h_j(u)| & \leq \int_{\S^{n-1}} |g_n(\pair{u}{v})| h_j(v)\,dv = \int_{C_{\alpha_j} \cup (-C_{\alpha_j})} |g_n(\pair{u}{v})| h_j(v)\,dv \\
			& \leq \int_{C_{\alpha_j} \cup (-C_{\alpha_j})} MC \,dv = 2MC \sigma(C_{\alpha_j}),
		\end{align*}
		where $\sigma$ denotes the spherical Lebesgue measure. Since this bound does not depend on $u \in A_t$ anymore, we can choose $j$ large enough, such that $\Box_n^{-1}h_j \leq \frac{1}{2}$ on $A_t$.
		
		Using that this construction works for every $t<1$, we will now choose $t$ in such a way that we can control the behavior of $\Box_n^{-1}h_j$ on $\S^{n-1}\setminus A_t$. To this end, choose $t_n < 1$ such that $g_n(t) \leq 0$ for all $t > t_n$. Such $t_n$ exists by \cite{Berg1969}*{Thm.~3.3(ii)}. Next, choose $t_n<t<1$ arbitrary. By the convergence $\alpha_j \to 0$ and the previous step, there exists $j_1 \in \N$ such that, for all $j \geq j_1$, $\Box_n^{-1}h_{j} \leq \frac{1}{2}$ on $A_t$ and for every $u \in \S^{n-1}\setminus A_t$ the spherical cap of radius $\alpha_{j}$ around $u$ does not intersect $A_{t_n}$.
		
		Then we can estimate for $j \geq j_1$ and $u \in \S^{n-1} \setminus A_t$ with $\pm \pair{u}{\bar e} > 0$,
		\begin{align*}
			\Box_n^{-1} h_j(u) & = \int_{\S^{n-1}} g_n(\pair{u}{v}) h_j(v)\,dv = \int_{C_{\alpha_j} \cup (-C_{\alpha_j})} g_n(\pair{u}{v}) h_j(v)\,dv \\
			& \leq \int_{\mp C_{\alpha_j}} g_n(\pair{u}{v}) h_j(v) \,dv \leq MC\sigma(C_{\alpha_j}).
		\end{align*}
		By the choice of $j$ from above (for $u \in A_t$, ensuring that $2MC \sigma(C_{\alpha_j}) \leq \frac{1}{2}$), we conclude that $\Box_n^{-1}h_j \leq \frac{1}{4}$ on $\S^{n-1}\setminus A_t$. In total, we have obtained that $1 - \Box_n^{-1} h_j \geq \frac{1}{2}>0$ on $\S^{n-1}$ for all sufficiently large $j \in \N$, concluding the proof.	
	\end{proof}
	
	
	\medskip
	
	In the remainder of the section, we discuss non-polar inequalities. Here, the argument from \cref{mthm:existExtr} (with small adaptations) applies directly and yields the following existence result. As the proof is almost verbatim, we will omit it.
	\begin{thm} \label{thm:existExtrNonPolar} Suppose that $\Phi_i: \K^n \rightarrow \K^n$ is a continuous, translation invariant, and $\SO(n)$ equivariant Minkowski valuation of a given degree
		$i \in \{1, \ldots, n - 1\}$. If the generating function of $\Phi_i$ is a sum of two generating functions one of which is bounded from below by a positive constant, then
		\begin{align}\label{eq:thmExtNonPolar}
			V_n(\Phi_i K)/V_n(K)^i
		\end{align}
		attains a minimum on convex bodies $K \in \K^n$ with non-empty interior.
	\end{thm}
	
	Additionally, a short proof yields the minimizers of \eqref{eq:thmExtNonPolar} among all convex bodies of revolution for weakly monotone Minkowski valuations of degree $1$. Here, we denote again by $s(K)$ the Steiner point of $K \in \K^n$, see \cite{Schneider2014}*{Sec.~1.7}.
	\begin{thm}
		Suppose that $\Phi:\K^n \rightarrow \K^n$ is a continuous, translation invariant and $\SO(n)$ equivariant, weakly monotone, non-trivial Minkowski valuation of degree $1$. Among convex bodies $K \in \K^n$ of revolution with non-empty interior,
		\begin{align*}
			\frac{V_n(\Phi K)}{V_n(K)}
		\end{align*}
		is minimized by Euclidean balls. If $\Phi K = aK + b(-K) - \alpha s(K)$ for some $a, b \geq 0$ and $\alpha \in \R$, then $K$ is a minimizer if and only if $K$ is centrally symmetric, for $a \neq 0 \neq b$, or $K$ is any convex body of revolution, for $a=0$ or $b=0$. Otherwise, Euclidean balls are the only minimizers.
	\end{thm}
	\begin{proof}
		By \cref{thm:kiderlenMonMinkEnd}, there exists an $\SO(n-1)$ invariant measure $\mu = \widetilde{\mu} + \alpha \pair{u}{\bar e}du$ on $\S^{n-1}$ such that $\widetilde{\mu}$ is non-negative and $\alpha \in \R$. As $h(K, \cdot) \ast (\pair{u}{\bar e}du)$ yields a constant multiple of the Steiner point, we therefore obtain
		\begin{align*}
			V_n(\Phi K) = \frac{1}{n} \int_{\S^{n-1}} (h(K, \cdot) \ast \widetilde{\mu})(u) dS_{n-1}(\Phi K, u)
		\end{align*}
		 for every $K \in \K^n$ with non-empty interior. If $K$ is additionally a body of revolution (without loss of generality we may assume that the axis is $\bar e$), then $h(K, \cdot)$ is $\SO(n-1)$ invariant. As trivially $\widehat{f} = f$ for every $\SO(n-1)$ invariant $f$, 
		 \begin{align*}
		 	h(K, \cdot) \ast \widetilde{\mu} = \widehat{h(K, \cdot)} \ast \widehat{(\widetilde{\mu})} = \widetilde{\mu} \ast h(K, \cdot),
		 \end{align*}
	 	 and thus, by Fubini's theorem and \eqref{eq:ConvSphMeasSph},
	 	 \begin{align*}
	 	 	V_n(\Phi K) &= \frac{1}{n}\int_{\S^{n-1}} \int_{\S^{n-1}} h(K, \vartheta_v^{-1}u)dS_{n-1}(\Phi K, u) d\widetilde{\mu}(v) \\
	 	 	&= \int_{\S^{n-1}} V(\vartheta_vK, \Phi K, \dots, \Phi K) d\widetilde{\mu}(v),
	 	 \end{align*}
 	 	 where $\Phi K$ appears $(n-1)$-times in the mixed volume. Applying Minkowski's inequality (see, e.g., \cite{Schneider2014}*{Ch.~7}), we obtain
	 	 \begin{align}\label{eq:prfMinkEndosZonalMinkIneq}
	 		V_n(\Phi K) \geq \widetilde{\mu}(\S^{n-1}) V_n(K)^{\frac{1}{n}} V_n(\Phi K)^{\frac{n-1}{n}}.
 		\end{align}
 		Since $\Phi B^n = \widetilde{\mu}(\S^{n-1})B^n$, the claimed inequality follows.
 		
 		For the equality cases, first note that we may assume that the Steiner points of $K$ (by translation-invariance) and $\Phi K$ (this is always true, see \cite{Kiderlen2006}) lie at the origin and that $\widetilde{\mu}$ is a probability measure (by scaling invariance). Inequality~\eqref{eq:prfMinkEndosZonalMinkIneq} therefore reads $V_n(\Phi K) \geq V_n(K)$ and the equality cases of Minkowski's inequality imply that $\Phi K = \lambda_v \vartheta_v K + x_v$ for $\widetilde{\mu}$-almost all $v \in \S^{n-1}$, and $\lambda_v > 0, x_v \in \R^n$. By above assumptions, $\lambda_v = 1$ and $x_v = 0$, that is, we have
 		\begin{align}\label{eq:prfMinkEndoZonalCondEq}
 			\Phi K = \vartheta_v K \quad \text{ for $\widetilde{\mu}$-a.e. } v \in \S^{n-1}.
 		\end{align}
 		As both sides are continuous in $v$, the statement holds for all $v \in \mathrm{supp}\, \widetilde{\mu}$.
 	
 		If $\widetilde{\mu}$ is discrete, by $\SO(n-1)$ invariance, it must be of the form $\widetilde{\mu} = a \delta_{\bar e} + b \delta_{-\bar e}$ with $a+b=1$, $a,b \geq 0$. Hence, \eqref{eq:prfMinkEndoZonalCondEq} reads for $a \neq 0 \neq b$
 		\begin{align*}
 			K = \vartheta_{\bar e} K = \Phi K = \vartheta_{-\bar e} = -K,
 		\end{align*}
 		that is, equality holds if and only if $K$ is origin-symmetric. Thus, by translation invariance, centrally symmetric bodies are the only minimizers.
 		
		For $a=0$ or $b=0$, we have that $\Phi K = \pm K$ for every $K \in \K^n$ and $V_n(\Phi K)/V_n(K)$ is constant.
		
		If $\widetilde{\mu}$ is not discrete, \eqref{eq:prfMinkEndoZonalCondEq} implies that $\vartheta_v K = \vartheta_w K$ for all $v,w \in \mathrm{supp}\, \widetilde{\mu}$. Consequently, by the $\SO(n-1)$ invariance of $h(K, \cdot)$,
		\begin{align*}
			h(\vartheta_u K, v) = h(\vartheta_v K, u) = h(\vartheta_w K, u) = h(\vartheta_u K, w), \quad u \in \S^{n-1},
		\end{align*}
		that is, $h(\vartheta_u K, \cdot)$ is constant on $\mathrm{supp}\,\widetilde{\mu}$ for all $u \in \S^{n-1}$. As $K$ is a body of revolution and, by $\SO(n-1)$ invariance, $\mathrm{supp}\,\widetilde{\mu}$ contains at least an (affine) sub-sphere, this is easily seen to imply that $K$ must be a ball, concluding the proof.
	\end{proof}
	\bigskip

	\subsection*{Acknowledgments}
	The authors were supported by the Austrian Science Fund (FWF), \href{https://doi.org/10.55776/P31448}{doi:10.55776/P31448}. The first-named author was additionally supported by the Austrian Science Fund (FWF), \href{https://doi.org/10.55776/P34446}{doi:10.55776/P34446}.
	
	\begingroup
	\let\itshape\upshape
	
	\bibliographystyle{abbrv}
	\bibliography{references}{}

\begin{bibdiv}
\begin{biblist}

\bib{Abardia2012b}{article}{
      author={Abardia, Judit},
       title={Difference bodies in complex vector spaces},
        date={2012},
        ISSN={0022-1236},
     journal={J. Funct. Anal.},
      volume={263},
      number={11},
       pages={3588\ndash 3603},
         url={https://doi.org/10.1016/j.jfa.2012.09.002},
      review={\MR{2984076}},
}

\bib{Abardia2011b}{article}{
      author={Abardia, Judit},
      author={Bernig, Andreas},
       title={Projection bodies in complex vector spaces},
        date={2011},
        ISSN={0001-8708},
     journal={Adv. Math.},
      volume={227},
      number={2},
       pages={830\ndash 846},
         url={https://doi.org/10.1016/j.aim.2011.02.013},
      review={\MR{2793024}},
}

\bib{Abardia2018b}{article}{
      author={Abardia-Ev\'{e}quoz, Judit},
      author={Colesanti, Andrea},
      author={Saor\'{\i}n~G\'{o}mez, Eugenia},
       title={Minkowski valuations under volume constraints},
        date={2018},
        ISSN={0001-8708},
     journal={Adv. Math.},
      volume={333},
       pages={118\ndash 158},
         url={https://doi.org/10.1016/j.aim.2018.05.033},
      review={\MR{3818074}},
}

\bib{Alesker2004b}{article}{
      author={Alesker, Semyon},
       title={The multiplicative structure on continuous polynomial
  valuations},
        date={2004},
        ISSN={1016-443X},
     journal={Geom. Funct. Anal.},
      volume={14},
      number={1},
       pages={1\ndash 26},
         url={https://doi.org/10.1007/s00039-004-0450-2},
      review={\MR{2053598}},
}

\bib{Alesker2006b}{article}{
      author={Alesker, Semyon},
       title={Theory of valuations on manifolds. {I}. {L}inear spaces},
        date={2006},
        ISSN={0021-2172},
     journal={Israel J. Math.},
      volume={156},
       pages={311\ndash 339},
         url={https://doi.org/10.1007/BF02773837},
      review={\MR{2282381}},
}

\bib{Alesker2006}{article}{
      author={Alesker, Semyon},
       title={Theory of valuations on manifolds. {II}},
        date={2006},
        ISSN={0001-8708},
     journal={Adv. Math.},
      volume={207},
      number={1},
       pages={420\ndash 454},
         url={https://doi.org/10.1016/j.aim.2005.11.015},
      review={\MR{2264077}},
}

\bib{Alesker2007}{incollection}{
      author={Alesker, Semyon},
       title={Theory of valuations on manifolds. {IV}. {N}ew properties of the
  multiplicative structure},
        date={2007},
   booktitle={Geometric aspects of functional analysis},
      series={Lecture Notes in Math.},
      volume={1910},
   publisher={Springer, Berlin},
       pages={1\ndash 44},
         url={https://doi.org/10.1007/978-3-540-72053-9_1},
      review={\MR{2347038}},
}

\bib{Alesker2009}{article}{
      author={Alesker, Semyon},
       title={Valuations on manifolds and integral geometry},
        date={2010},
        ISSN={1016-443X},
     journal={Geom. Funct. Anal.},
      volume={20},
      number={5},
       pages={1073\ndash 1143},
         url={https://doi.org/10.1007/s00039-010-0088-1},
      review={\MR{2746948}},
}

\bib{Alesker2011b}{article}{
      author={Alesker, Semyon},
       title={A {F}ourier-type transform on translation-invariant valuations on
  convex sets},
        date={2011},
        ISSN={0021-2172},
     journal={Israel J. Math.},
      volume={181},
       pages={189\ndash 294},
         url={https://doi.org/10.1007/s11856-011-0008-6},
      review={\MR{2773042}},
}

\bib{Alesker2011}{article}{
      author={Alesker, Semyon},
      author={Bernig, Andreas},
      author={Schuster, Franz~E.},
       title={Harmonic analysis of translation invariant valuations},
        date={2011},
        ISSN={1016-443X},
     journal={Geom. Funct. Anal.},
      volume={21},
      number={4},
       pages={751\ndash 773},
         url={https://doi.org/10.1007/s00039-011-0125-8},
      review={\MR{2827009}},
}

\bib{Alesker2004}{article}{
      author={Alesker, Semyon},
      author={Bernstein, Joseph},
       title={Range characterization of the cosine transform on higher
  {G}rassmannians},
        date={2004},
        ISSN={0001-8708,1090-2082},
     journal={Adv. Math.},
      volume={184},
      number={2},
       pages={367\ndash 379},
         url={https://doi.org/10.1016/S0001-8708(03)00149-X},
      review={\MR{2054020}},
}

\bib{Alesker2014}{article}{
      author={Alesker, Semyon},
      author={Faifman, Dmitry},
       title={Convex valuations invariant under the {L}orentz group},
        date={2014},
        ISSN={0022-040X},
     journal={J. Differential Geom.},
      volume={98},
      number={2},
       pages={183\ndash 236},
         url={http://projecteuclid.org/euclid.jdg/1406552249},
      review={\MR{3238311}},
}

\bib{Alesker2008}{article}{
      author={Alesker, Semyon},
      author={Fu, Joseph H.~G.},
       title={Theory of valuations on manifolds. {III}. {M}ultiplicative
  structure in the general case},
        date={2008},
        ISSN={0002-9947},
     journal={Trans. Amer. Math. Soc.},
      volume={360},
      number={4},
       pages={1951\ndash 1981},
         url={https://doi.org/10.1090/S0002-9947-07-04489-3},
      review={\MR{2366970}},
}

\bib{Berg2014}{article}{
      author={Berg, Astrid},
      author={Parapatits, Lukas},
      author={Schuster, Franz~E.},
      author={Weberndorfer, Manuel},
       title={Log-concavity properties of {M}inkowski valuations},
        date={2018},
        ISSN={0002-9947},
     journal={Trans. Amer. Math. Soc.},
      volume={370},
      number={7},
       pages={5245\ndash 5277},
         url={https://doi.org/10.1090/tran/7434},
        note={With an appendix by Semyon Alesker},
      review={\MR{3787383}},
}

\bib{Berg2020}{article}{
      author={Berg, Astrid},
      author={Schuster, Franz~E.},
       title={Lutwak-{P}etty projection inequalities for {M}inkowski valuations
  and their duals},
        date={2020},
        ISSN={0022-247X},
     journal={J. Math. Anal. Appl.},
      volume={490},
      number={1},
       pages={124190},
         url={https://doi.org/10.1016/j.jmaa.2020.124190},
      review={\MR{4097946}},
}

\bib{Berg1969}{article}{
      author={Berg, Christian},
       title={Corps convexes et potentiels sph\'{e}riques},
        date={1969},
        ISSN={0023-3323},
     journal={Mat.-Fys. Medd. Danske Vid. Selsk.},
      volume={37},
      number={6},
       pages={64 pp. (1969)},
      review={\MR{254789}},
}

\bib{Brauner2023}{article}{
      author={Brauner, Leo},
      author={Ortega-Moreno, Oscar},
       title={Fixed points of mean section operators},
        date={2024},
     journal={Trans. Amer. Math. Soc. (in press)},
      eprint={arXiv:2302.11973},
}

\bib{Dafnis2012}{article}{
      author={Dafnis, Nikos},
      author={Paouris, Grigoris},
       title={Estimates for the affine and dual affine quermassintegrals of
  convex bodies},
        date={2012},
        ISSN={0019-2082},
     journal={Illinois J. Math.},
      volume={56},
      number={4},
       pages={1005\ndash 1021},
         url={http://projecteuclid.org/euclid.ijm/1399395821},
      review={\MR{3231472}},
}

\bib{Dorrek2017b}{article}{
      author={Dorrek, Felix},
       title={Minkowski endomorphisms},
        date={2017},
        ISSN={1016-443X},
     journal={Geom. Funct. Anal.},
      volume={27},
      number={3},
       pages={466\ndash 488},
         url={https://doi.org/10.1007/s00039-017-0405-z},
      review={\MR{3655954}},
}

\bib{Dorrek2017}{article}{
      author={Dorrek, Felix},
      author={Schuster, Franz~E.},
       title={Projection functions, area measures and the {A}lesker-{F}ourier
  transform},
        date={2017},
        ISSN={0022-1236},
     journal={J. Funct. Anal.},
      volume={273},
      number={6},
       pages={2026\ndash 2069},
         url={https://doi.org/10.1016/j.jfa.2017.06.003},
      review={\MR{3669029}},
}

\bib{Esposito2005}{article}{
      author={Esposito, Luca},
      author={Fusco, Nicola},
      author={Trombetti, Cristina},
       title={A quantitative version of the isoperimetric inequality: the
  anisotropic case},
        date={2005},
        ISSN={0391-173X},
     journal={Ann. Sc. Norm. Super. Pisa Cl. Sci. (5)},
      volume={4},
      number={4},
       pages={619\ndash 651},
      review={\MR{2207737}},
}

\bib{Gardner2006}{book}{
      author={Gardner, Richard~J.},
       title={Geometric tomography},
     edition={Second},
      series={Encyclopedia of Mathematics and its Applications},
   publisher={Cambridge University Press, New York},
        date={2006},
      volume={58},
        ISBN={0-521; 0-521-68493-5},
         url={https://doi.org/10.1017/CBO9781107341029},
      review={\MR{2251886}},
}

\bib{Grinberg1999}{article}{
      author={Grinberg, Eric},
      author={Zhang, Gaoyong},
       title={Convolutions, transforms, and convex bodies},
        date={1999},
        ISSN={0024-6115},
     journal={Proc. London Math. Soc. (3)},
      volume={78},
      number={1},
       pages={77\ndash 115},
         url={https://doi.org/10.1112/S0024611599001653},
      review={\MR{1658156}},
}

\bib{Grinberg1991}{article}{
      author={Grinberg, Eric~L.},
       title={Isoperimetric inequalities and identities for {$k$}-dimensional
  cross-sections of convex bodies},
        date={1991},
        ISSN={0025-5831},
     journal={Math. Ann.},
      volume={291},
      number={1},
       pages={75\ndash 86},
         url={https://doi.org/10.1007/BF01445191},
      review={\MR{1125008}},
}

\bib{Gritzmann1987}{article}{
      author={Gritzmann, P.},
      author={Wills, J.~M.},
      author={Wrase, D.},
       title={A new isoperimetric inequality},
        date={1987},
        ISSN={0075-4102},
     journal={J. Reine Angew. Math.},
      volume={379},
       pages={22\ndash 30},
      review={\MR{903632}},
}

\bib{Gritzmann1988}{article}{
      author={Gritzmann, Peter},
       title={A characterization of all loglinear inequalities for three
  quermassintegrals of convex bodies},
        date={1988},
        ISSN={0002-9939},
     journal={Proc. Amer. Math. Soc.},
      volume={104},
      number={2},
       pages={563\ndash 570},
         url={https://doi.org/10.2307/2047012},
      review={\MR{962829}},
}

\bib{Groemer1996}{book}{
      author={Groemer, H.},
       title={Geometric applications of {F}ourier series and spherical
  harmonics},
      series={Encyclopedia of Mathematics and its Applications},
   publisher={Cambridge University Press, Cambridge},
        date={1996},
      volume={61},
        ISBN={0-521-47318-7},
         url={https://doi.org/10.1017/CBO9780511530005},
      review={\MR{1412143}},
}

\bib{Haberl2019}{article}{
      author={Haberl, Christoph},
      author={Schuster, Franz~E.},
       title={Affine vs. {E}uclidean isoperimetric inequalities},
        date={2019},
        ISSN={0001-8708},
     journal={Adv. Math.},
      volume={356},
       pages={106811, 26 pp.},
         url={https://doi.org/10.1016/j.aim.2019.106811},
      review={\MR{4008006}},
}

\bib{Hadwiger1957}{book}{
      author={Hadwiger, H.},
       title={Vorlesungen \"{u}ber {I}nhalt, {O}berfl\"{a}che und
  {I}soperimetrie},
   publisher={Springer-Verlag, Berlin-G\"{o}ttingen-Heidelberg},
        date={1957},
      review={\MR{0102775}},
}

\bib{Hofstaetter2021}{article}{
      author={Hofst\"{a}tter, Georg~C.},
      author={Schuster, Franz~E.},
       title={Blaschke-{S}antal\'{o} inequalities for {M}inkowski and {A}splund
  endomorphisms},
        date={2023},
        ISSN={1073-7928},
     journal={Int. Math. Res. Not. IMRN},
      number={2},
       pages={1378\ndash 1419},
         url={https://doi.org/10.1093/imrn/rnab262},
      review={\MR{4537328}},
}

\bib{Kiderlen2006}{article}{
      author={Kiderlen, Markus},
       title={Blaschke- and {M}inkowski-endomorphisms of convex bodies},
        date={2006},
        ISSN={0002-9947},
     journal={Trans. Amer. Math. Soc.},
      volume={358},
      number={12},
       pages={5539\ndash 5564},
         url={https://doi.org/10.1090/S0002-9947-06-03914-6},
      review={\MR{2238926}},
}

\bib{Klain2000}{article}{
      author={Klain, Daniel~A.},
       title={Even valuations on convex bodies},
        date={2000},
        ISSN={0002-9947},
     journal={Trans. Amer. Math. Soc.},
      volume={352},
      number={1},
       pages={71\ndash 93},
         url={https://doi.org/10.1090/S0002-9947-99-02240-0},
      review={\MR{1487620}},
}

\bib{Klain1997}{book}{
      author={Klain, Daniel~A.},
      author={Rota, Gian-Carlo},
       title={Introduction to geometric probability},
      series={Lezioni Lincee. [Lincei Lectures]},
   publisher={Cambridge University Press, Cambridge},
        date={1997},
        ISBN={0-521-59362-X; 0-521-59654-8},
      review={\MR{1608265}},
}

\bib{Ludwig2002b}{article}{
      author={Ludwig, Monika},
       title={Projection bodies and valuations},
        date={2002},
        ISSN={0001-8708},
     journal={Adv. Math.},
      volume={172},
      number={2},
       pages={158\ndash 168},
         url={https://doi.org/10.1016/S0001-8708(02)00021-X},
      review={\MR{1942402}},
}

\bib{Ludwig2005}{article}{
      author={Ludwig, Monika},
       title={Minkowski valuations},
        date={2005},
        ISSN={0002-9947},
     journal={Trans. Amer. Math. Soc.},
      volume={357},
      number={10},
       pages={4191\ndash 4213},
         url={https://doi.org/10.1090/S0002-9947-04-03666-9},
      review={\MR{2159706}},
}

\bib{Ludwig2010b}{article}{
      author={Ludwig, Monika},
       title={Minkowski areas and valuations},
        date={2010},
        ISSN={0022-040X},
     journal={J. Differential Geom.},
      volume={86},
      number={1},
       pages={133\ndash 161},
         url={http://projecteuclid.org/euclid.jdg/1299766685},
      review={\MR{2772547}},
}

\bib{Lutwak1975}{article}{
      author={Lutwak, Erwin},
       title={A general {B}ieberbach inequality},
        date={1975},
        ISSN={0305-0041},
     journal={Math. Proc. Cambridge Philos. Soc.},
      volume={78},
      number={3},
       pages={493\ndash 495},
         url={https://doi.org/10.1017/S0305004100051963},
      review={\MR{0383253}},
}

\bib{Lutwak1984}{article}{
      author={Lutwak, Erwin},
       title={A general isepiphanic inequality},
        date={1984},
        ISSN={0002-9939},
     journal={Proc. Amer. Math. Soc.},
      volume={90},
      number={3},
       pages={415\ndash 421},
         url={https://doi.org/10.2307/2044485},
      review={\MR{728360}},
}

\bib{Lutwak1985b}{article}{
      author={Lutwak, Erwin},
       title={Mixed projection inequalities},
        date={1985},
        ISSN={0002-9947},
     journal={Trans. Amer. Math. Soc.},
      volume={287},
      number={1},
       pages={91\ndash 105},
         url={https://doi.org/10.2307/2000399},
      review={\MR{766208}},
}

\bib{Lutwak1988b}{article}{
      author={Lutwak, Erwin},
       title={Inequalities for {H}adwiger's harmonic {Q}uermassintegrals},
        date={1988},
        ISSN={0025-5831},
     journal={Math. Ann.},
      volume={280},
      number={1},
       pages={165\ndash 175},
         url={https://doi.org/10.1007/BF01474188},
      review={\MR{928304}},
}

\bib{McMullen1977}{article}{
      author={McMullen, P.},
       title={Valuations and {E}uler-type relations on certain classes of
  convex polytopes},
        date={1977},
        ISSN={0024-6115},
     journal={Proc. London Math. Soc. (3)},
      volume={35},
      number={1},
       pages={113\ndash 135},
         url={https://doi.org/10.1112/plms/s3-35.1.113},
      review={\MR{448239}},
}

\bib{Milman2020}{article}{
      author={Milman, Emanuel},
      author={Yehudayoff, Amir},
       title={Sharp isoperimetric inequalities for affine quermassintegrals},
        date={2023},
        ISSN={0894-0347,1088-6834},
     journal={J. Amer. Math. Soc.},
      volume={36},
      number={4},
       pages={1061\ndash 1101},
         url={https://doi.org/10.1090/jams/1013},
      review={\MR{4618955}},
}

\bib{OrtegaMoreno2021}{article}{
      author={Ortega-Moreno, Oscar},
      author={Schuster, Franz~E.},
       title={Fixed points of {M}inkowski valuations},
        date={2021},
        ISSN={0001-8708},
     journal={Adv. Math.},
      volume={392},
       pages={Paper No. 108017, 33},
         url={https://doi.org/10.1016/j.aim.2021.108017},
      review={\MR{4316669}},
}

\bib{Paouris2013}{article}{
      author={Paouris, Grigoris},
      author={Pivovarov, Peter},
       title={Small-ball probabilities for the volume of random convex sets},
        date={2013},
        ISSN={0179-5376},
     journal={Discrete Comput. Geom.},
      volume={49},
      number={3},
       pages={601\ndash 646},
         url={https://doi.org/10.1007/s00454-013-9492-2},
      review={\MR{3038532}},
}

\bib{Parapatits2012}{article}{
      author={Parapatits, Lukas},
      author={Schuster, Franz~E.},
       title={The {S}teiner formula for {M}inkowski valuations},
        date={2012},
        ISSN={0001-8708},
     journal={Adv. Math.},
      volume={230},
      number={3},
       pages={978\ndash 994},
         url={https://doi.org/10.1016/j.aim.2012.03.024},
      review={\MR{2921168}},
}

\bib{Petty1971}{inproceedings}{
      author={Petty, C.~M.},
       title={Isoperimetric problems},
        date={1971},
   booktitle={Proceedings of the {C}onference on {C}onvexity and
  {C}ombinatorial {G}eometry ({U}niv. {O}klahoma, {N}orman, {O}kla., 1971)},
   publisher={Dept. Math., Univ. Oklahoma, Norman, Okla.},
       pages={26\ndash 41},
      review={\MR{0362057}},
}

\bib{Schneider1970}{article}{
      author={Schneider, Rolf},
       title={\"{U}ber eine {I}ntegralgleichung in der {T}heorie der konvexen
  {K}\"{o}rper},
        date={1970},
        ISSN={0025-584X},
     journal={Math. Nachr.},
      volume={44},
       pages={55\ndash 75},
         url={https://doi.org/10.1002/mana.19700440105},
      review={\MR{275286}},
}

\bib{Schneider2014}{book}{
      author={Schneider, Rolf},
       title={Convex bodies: the {B}runn-{M}inkowski theory},
     edition={expanded},
      series={Encyclopedia of Mathematics and its Applications},
   publisher={Cambridge University Press, Cambridge},
        date={2014},
      volume={151},
        ISBN={978-1-107-60101-7},
      review={\MR{3155183}},
}

\bib{Schneider2008}{book}{
      author={Schneider, Rolf},
      author={Weil, Wolfgang},
       title={Stochastic and integral geometry},
      series={Probability and its Applications (New York)},
   publisher={Springer-Verlag, Berlin},
        date={2008},
        ISBN={978-3-540-78858-4},
         url={https://doi.org/10.1007/978-3-540-78859-1},
      review={\MR{2455326}},
}

\bib{Schuster2006}{article}{
      author={Schuster, Franz~E.},
       title={Volume inequalities and additive maps of convex bodies},
        date={2006},
        ISSN={0025-5793},
     journal={Mathematika},
      volume={53},
      number={2},
       pages={211\ndash 234 (2007)},
         url={https://doi.org/10.1112/S0025579300000103},
      review={\MR{2343256}},
}

\bib{Schuster2007}{article}{
      author={Schuster, Franz~E.},
       title={Convolutions and multiplier transformations of convex bodies},
        date={2007},
        ISSN={0002-9947},
     journal={Trans. Amer. Math. Soc.},
      volume={359},
      number={11},
       pages={5567\ndash 5591},
         url={https://doi.org/10.1090/S0002-9947-07-04270-5},
      review={\MR{2327043}},
}

\bib{Schuster2010}{article}{
      author={Schuster, Franz~E.},
       title={Crofton measures and {M}inkowski valuations},
        date={2010},
        ISSN={0012-7094},
     journal={Duke Math. J.},
      volume={154},
      number={1},
       pages={1\ndash 30},
         url={https://doi.org/10.1215/00127094-2010-033},
      review={\MR{2668553}},
}

\bib{Schuster2012}{article}{
      author={Schuster, Franz~E.},
      author={Wannerer, Thomas},
       title={{${\rm GL}(n)$} contravariant {M}inkowski valuations},
        date={2012},
        ISSN={0002-9947},
     journal={Trans. Amer. Math. Soc.},
      volume={364},
      number={2},
       pages={815\ndash 826},
         url={https://doi.org/10.1090/S0002-9947-2011-05364-X},
      review={\MR{2846354}},
}

\bib{Schuster2015}{article}{
      author={Schuster, Franz~E.},
      author={Wannerer, Thomas},
       title={Even {M}inkowski valuations},
        date={2015},
        ISSN={0002-9327},
     journal={Amer. J. Math.},
      volume={137},
      number={6},
       pages={1651\ndash 1683},
         url={https://doi.org/10.1353/ajm.2015.0041},
      review={\MR{3432270}},
}

\bib{Schuster2018}{article}{
      author={Schuster, Franz~E.},
      author={Wannerer, Thomas},
       title={Minkowski valuations and generalized valuations},
        date={2018},
        ISSN={1435-9855},
     journal={J. Eur. Math. Soc. (JEMS)},
      volume={20},
      number={8},
       pages={1851\ndash 1884},
         url={https://doi.org/10.4171/JEMS/801},
      review={\MR{3854893}},
}

\bib{Wannerer2011}{article}{
      author={Wannerer, Thomas},
       title={{${\rm GL}(n)$} equivariant {M}inkowski valuations},
        date={2011},
        ISSN={0022-2518},
     journal={Indiana Univ. Math. J.},
      volume={60},
      number={5},
       pages={1655\ndash 1672},
         url={https://doi.org/10.1512/iumj.2011.60.4425},
      review={\MR{2997003}},
}

\end{biblist}
\end{bibdiv}
	
	\endgroup

\end{document}